\documentclass[12 pt,reqno]{amsart}
\usepackage{amssymb,amsmath,amstext,amsgen,amsfonts,amsthm,verbatim,hyperref,fullpage,enumerate}
\usepackage[usenames,dvipsnames]{color}
\usepackage{amsrefs}

\newcommand{\N}{\mathbb{N}}
\newcommand{\R}{\mathbb{R}}

\newcommand{\Z}{\mathbb{Z}}

\renewcommand{\epsilon}{\varepsilon}
\renewcommand{\phi}{\varphi}
\renewcommand{\hat}{\widehat}
\renewcommand{\tilde}{\widetilde}

\newcommand{\SI}{\mathcal{SI}}
\newcommand{\LD}{\mathcal{LD}}
\newcommand{\M}{\mathcal{M}}
\newcommand{\Hd}{\mathcal{H}}
\newcommand{\calB}{\mathcal{B}}
\newcommand{\calE}{\mathcal{E}}
\newcommand{\calF}{\mathcal{F}}
\newcommand{\calG}{\mathcal{G}}
\newcommand{\calX}{\mathcal{X}}
\newcommand{\dist}{\mathrm{dist}}
\DeclareMathOperator{\diam}{diam}
\DeclareMathOperator{\dom}{dom}
\DeclareMathOperator{\im}{Image}
\DeclareMathOperator{\Lip}{Lip}

\newtheorem{theorem}{Theorem}[section]
\newtheorem{proposition}[theorem]{Proposition}
\newtheorem{lemma}[theorem]{Lemma}
\newtheorem{corollary}[theorem]{Corollary}

\newtheorem{innercustomlemma}{Lemma}
\newenvironment{customlemma}[1]
  {\renewcommand\theinnercustomlemma{#1}\innercustomlemma}
  {\endinnercustomlemma}

\newtheorem{innercustomproposition}{Proposition}
\newenvironment{customproposition}[1]
  {\renewcommand\theinnercustomproposition{#1}\innercustomproposition}
  {\endinnercustomproposition}

\newtheorem{innercustomclaim}{Claim}
\newenvironment{customclaim}[1]
  {\renewcommand\theinnercustomclaim{#1}\innercustomclaim}
  {\endinnercustomclaim}

\newtheorem{innercustomreduction}{Reduction}
\newenvironment{customreduction}[1]
  {\renewcommand\theinnercustomreduction{#1}\innercustomreduction}
  {\endinnercustomreduction}

\theoremstyle{remark}
\newtheorem{remark}[theorem]{Remark}

\begin{document}
\title{Characterizations of rectifiable metric measure spaces}

%\title{Rectifiability of Lipschitz differentiability spaces}
\author{David Bate \and Sean Li}
\date{\today}
\address{Department of Mathematics, The University of Chicago, Chicago, IL 60637}
\email{bate@math.uchicago.edu}
\email{seanli@math.uchicago.edu}

\begin{abstract}
  We characterize $n$-rectifiable metric measure spaces as those spaces that admit a countable Borel decomposition so that each piece has positive and finite $n$-densities and one of the following: is an $n$-dimensional Lipschitz differentiability space; has $n$-independent Alberti representations; satisfies David's condition for an $n$-dimensional chart.  The key tool is an iterative grid construction which allows us to show that the image of a ball with a high density of curves from the Alberti representations under a chart map contains a large portion of a uniformly large ball and hence satisfies David's condition.  This allows us to apply modified versions of previously known ``biLipschitz pieces'' results \cite{david,jones,david-semmes,semmes} on the charts.
\end{abstract}

%\begin{abstract}
%  We prove that $n$-dimensional charts in any Lipschitz differentiability space with positive and finite $n$-densities are $n$-rectifiable.  As the converse is true, this provides a characterization of rectifiable metric spaces in terms of density and differentiability.  This is done by showing that the chart map can be broken into biLipschitz pieces.  The key tool is an iterative grid construction which allows us to show that a the image of a ball with a high density of dense and long fragmented curves under a chart contains a large portion of a uniformly large ball.  This allows us to apply previously known ``biLipschitz pieces'' results \cite{david,jones,david-semmes,semmes} on the charts.
%\end{abstract}

\maketitle

\section{Introduction}
A metric measure space $(X,d,\mu)$ is said to be $n$-rectifiable if there exists a countable family of Lipschitz functions $f_i$ defined on measurable subsets $A_i \subset \R^n$ such that $\mu(X\backslash \bigcup_{i=1}^\infty f(A_i)) = 0$ and $\mu \ll \Hd^n$.  Similarly to how rectifiable subsets of Euclidean space possess many nice properties akin to those of smooth manifolds, rectifiable metric measure spaces also satisfy many regularity properties.  For these reasons, it is highly desirable to find general conditions that describe when a metric measure space is rectifiable.

Such conditions have been difficult to find.  Classically (that is, when the measure is defined on Euclidean space), this problem was first fully solved by Mattila \cite{mattila} for Hausdorff measure and more generally by Preiss \cite{preiss} for an arbitrary Radon measure.  In these results it was shown that the space in $n$-rectifiable if and only if the $n$-dimensional upper and lower densities of $\mu$
\begin{align*}
  \Theta^{*,n}(\mu;x) = \limsup_{r \to 0} \frac{\mu(B(x,r))}{(2r)^n}, \qquad \Theta^n_*(\mu;x) = \liminf_{r \to 0} \frac{\mu(B(x,r))}{(2r)^n}.
\end{align*}
agree and equal 1 (Mattila) or agree and are positive and finite (Preiss) at almost every point.  When these two values agree, we denote the common value by $\Theta^{n}(\mu;x)$, the $n$-dimensional density of $\mu$ at $x$.

In the metric setting, only partial answers are known.  We first partially  recall a theorem of Kirchheim \cite{kirchheim} that will be fundamental to our characterization of rectifiable metric measure spaces.
\begin{theorem}[Kirchheim \cite{kirchheim}]\label{t:kirchheim}
  Let $(X,d)$ be an $n$-rectifiable metric space of finite $\Hd^n$ measure. Then $\Theta^n(\Hd^n;x)=1$ for $\Hd^n$ a.e. $x$.
\end{theorem}
Conversely, Preiss and Ti\v{s}er \cite{preiss-tiser} showed that, for dimension 1, a lower Hausdorff density greater than (a number slightly less than) $3/4$ is sufficient to ensure 1-rectifiability.

In a different direction, but with a similar goal, the work of David and Semmes \cite{david,semmes,david-semmes} found conditions under which a Lipschitz function defined on an Ahlfors $n$-regular space taking values in $\R^n$ is in fact biLipschitz.  A metric measure space is said to be Ahlfors $n$-regular if there exist $C>0$ such that
\[\frac{1}{C}r^n \leq \mu(B(x,r)) \leq Cr^n\]
for each $x\in X$ and $r < \diam(X)$.  One may consider this condition to be a quantitative version of the property that the upper and lower $n$-densities are positive and finite.  David and Semmes showed that, if the image of every ball under a Lipschitz function contains a ball of comparable radius centred at the image of centre of the first ball (that is, the function is a Lipschitz quotient), then the function can be decomposed into biLipschitz pieces.  In fact, they only require that the image contains most (in terms of measure) of a ball of comparable radius.  We will discuss generalisations of this condition, now known as David's condition, below.  In particular, our main tool of constructing rectifiable subsets of metric measure spaces will be a generalisation of the theorems of David and Semmes that are applicable when the space only satisfies density estimates, rather than full Ahlfors regularity.  See Theorem \ref{t:general-decomp} for the statement of the generalization.

More recently, initiated by the striking work of Cheeger \cite{cheeger}, there has been much activity in generalising the classical theorem of Rademacher to metric measure spaces.  Most of all, this departed from the existing generalisations of Rademacher's theorem, for example that of Pansu \cite{pansu}, by not requiring a group structure in the domain to define the derivative.  Let $U \subset X$ be a Borel set and $\phi \colon X \to \mathbb R^n$ Lipschitz.  We say that $(U, \phi)$ form a chart of dimension $n$ and that a function $f : X \to \R$ is differentiable at $x_0 \in U$ with respect to $(U,\phi)$ if there exists a {\it unique} $Df(x_0) \in \R^n$ (the derivative of $f$ at $x_0$) so that
\begin{align*}
  \lim_{X\ni x \to x_0} \frac{|f(x) - f(x_0) - Df(x_0) \cdot (\phi(x) - \phi(x_0))|}{d(x,x_0)} = 0.
\end{align*}
A metric measure space $(X,d,\mu)$ is said to be a Lipschitz differentiability space if there exists a countable set of charts $(U_i,\phi_i)$ so that $X = \bigcup_i U_i$ and every real valued Lipschitz function defined on $X$ is differentiable at almost every point of every chart.  A Lipschitz differentiability space is said to be $n$-dimensional if every chart map is $\R^n$ valued.

Although the concept of a Lipschitz differentiability space is very general, it is known that some additional structure must exist---it must be possible to decompose the measure into an integral combination of 1-rectifiable measures known as Alberti representations.  Define $\Gamma$ to be the collection of biLipschitz functions $\gamma$ defined on a compact subset of $\R$ taking values in $X$ (known as curve fragments).  We say that $\mu$ has an Alberti representation if there exists a probability measure $\mathbb P$ on $\Gamma$ and measures $\mu_\gamma \ll \Hd^1 \llcorner \im \gamma$ such that
\[\mu(B) = \int_\Gamma \mu_\gamma(B) ~\text{d}\mathbb P\]
for each Borel set $B\subset X$.  Further, for a Lipschitz function $\phi \colon X \to \R^n$, an Alberti representation is in the $\phi$-direction of a cone $C \subset \R^n$ if
\[(\phi\circ\gamma)'(t) \in C\setminus \{0\},\]
for $\mathbb P$ a.e. $\gamma \in \Gamma$ and $\mu_\gamma$ a.e. $t \in \dom \gamma$.  Finally, we say that a collection of $n$ Alberti representations are $\phi$-independent if there exist linearly independent cones in $\R^n$ so that each Alberti representation is in the direction of a cone and no two Alberti representations are in the direction of the same cone.  One of the main results of \cite{bate} is that any $n$-dimensional chart in a Lipschitz differentiability space has $n$ independent Alberti representations.

There are known examples of Lipschitz differentiability spaces that are not groups (cf. \cite{bourdon-pajot,cheeger-kleiner,laakso}) and do not admit any rectifiable behaviour beyond the existence of Alberti representations.  Indeed, Cheeger also showed that for many of these spaces to possess a biLipschitz embedding into any Euclidean space, the dimension of the chart must equal the Hausdorff dimension of the space, which is not generally true.  (More generally, a theorem of Cheeger-Kleiner \cite{cheeger-kleiner-rnp} proves the same result for a biLipschitz embedding into an RNP Banach space.)  However, there are very natural relationships between rectifiability and differentiability.  For example, Rademacher's theorem easily extends to rectifiable sets via composition of functions (this concept is at the heart of the relationship between Alberti representations and differentiability).  Further, the proof of Kirchheim's theorem fundamentally relies on a version of Rademacher's theorem for metric space valued Lipschitz functions defined on Euclidean space.  Once this is established, the outline is that of the classical case.

In this paper we give precise conditions when the notions of rectifiability and differentiability agree and hence obtain several characterizations of rectifiable metric measure spaces.  Specifically, we prove the following theorem.

\begin{theorem}\label{t:main}
  A metric measure space $(X,d,\mu)$ is $n$-rectifiable (which we denote by Property (R)) if and only if there exist a countable collection of Borel sets $U_i\subset X$ with $\mu\left(X\setminus\bigcup_i U_i\right)=0$ and Lipschitz functions $\phi_i \colon X \to \R^n$ such that
  \[0< \Theta_*^n(\mu\llcorner U_i;x) \leq  \Theta^{*,n}(\mu\llcorner U_i;x) < \infty\]
  for $\mu$ a.e. $x \in U_i$ and every $i$, and one of the following holds:
  \begin{enumerate}[(i)]
    \item Each $(U_i,d, \mu)$ is an $n$-dimensional Lipschitz differentiability space;\label{i:main-diffspace}
    \item Each $\mu \llcorner U_i$ has $n$ $\phi_i$ independent Alberti representations;\label{i:main-alberti}
    \item Each $U_i$ satisfies David's condition a.e. with respect to the function $\phi_i$.\label{i:main-david}
  \end{enumerate}
\end{theorem}

See Section \ref{s:preliminaries} for the precise definition of David's condition that we use.  We note that, for each condition, a different countable decomposition is permitted.  Further, in the case that the space is rectifiable, the $U_i$ can be chosen to have finite $\Hd^n$ measure and so the Hausdorff density exists and equals 1 at almost every point by Kirchheim's theorem.  Therefore, the $\mu$ density must also exist and equals the Radon-Nikodym derivative of $\mu$ with respect to $\Hd^n$ at $\mu$ almost every point.

From this theorem we see that the density estimates force the curve fragments obtained from Alberti representations to form rectifiable sets.  This is not necessarily true without this condition.  For example, the Heisenberg group is a Lipschitz differentiability space consisting of a single chart of dimension two and hence has two independent Alberti representations but is purely 2-unrectifiable.  (The Heisenberg group is Ahlfors 4-regular.)  Furthermore, in \cite{mathe}, M\'ath\'e constructs a measure on $\R^3$ with $2$ independent Alberti representations that is purely $2$-unrectifiable.  Indeed, our results are new even for a measure defined on Euclidean space: the only existing result is a corollary of Cs\"ornyei-Jones \cite{csornyei-jones} that states that a measure on $\R^n$ with $n$-independent Alberti representations must be absolutely continuous with respect to Lebesgue measure.

We point out the known and easy implications in our theorem.  Firstly, as mentioned above, by Kirchheim's theorem we may take a countable decomposition of any rectifiable metric measure space so that the $n$-dimensional density of $\mu$ on each piece exists and is finite and positive at almost every point of such a piece.  Also from Kirchheim \cite{kirchheim}, we may suppose that the Lipschitz functions that parametrise a rectifiable set are in fact biLipschitz.  Further, it is easy to see that $f(A)$, for $f\colon A \subset \R^n \to f(A)$ biLipschitz and $A$ closed, equipped with $\mu\ll \Hd^n$ is a Lipschitz differentiability space (the chart map is simply the inverse of $f$), so that (R) implies \eqref{i:main-diffspace}.  Secondly, the existence of $n$ independent Alberti representations of an $n$ dimensional chart in a Lipschitz differentiability space is proved in \cite[Theorem 9.5]{bate}.  Thus, within this paper we are interested in proving \eqref{i:main-alberti} implies \eqref{i:main-david} and that \eqref{i:main-david} implies (R). 

%One obvious question to ask is how the charts give a differentiability structure to $X$.  In particular, if one can invert the $\phi_i$ in a Lipschitz manner, then the metric space is rectifiable and so the differential structure comes from the parameterization.  This is not always true as nonabelian Carnot groups admit charts but are not rectifiable.

An immediate corollary of our theorem is the following.

\begin{corollary}
  Any Ahlfors $n$-regular $n$-dimensional Lipschitz differentiability space is $n$-rectifiable.
\end{corollary}

This is a generalisation of a result of G. C. David (this David is not the same person who is the namesake of David's condition) in \cite{david-gc} where it is shown that tangents of almost every point of an $n$-dimensional chart in an Ahlfors $n$-regular Lipschitz differentiability spaces are $n$-rectifiable (in fact he showed even more that they are uniformly rectifiable).  In addition, he showed that $k$-dimensional charts in Ahlfors $s$-regular Lipschitz differentiability spaces for $k < s$ are strongly $k$-unrectifiable.  However, the rectifiability of Ahlfors $n$-regular Lipschitz differentiability spaces themselves was still unknown.  Indeed, this was the starting point for our work presented in this paper.

The proof of Theorem \ref{t:main} will follow the same outline as that in \cite{david-gc} although we will need to make all the arguments effective.  We quickly go over the general idea.  There, for any tangent $Y$ at a point $x \in X$, consider a limiting Lipschitz function $f \colon Y \to \R^n$ of $\phi$.  To show that $Y$ is rectifiable, it is shown that $f$ is a Lipschitz quotient and hence satisfies David's condition.  Thus, by applying the machinery developed by David \cite{david}, it is concluded that $f$ can be decomposed into biLipschitz pieces.

We will seek to use the same biLipschitz decomposition machinery and to do so, we need to show that images of balls contain large portions of uniformly large balls.  G. C. David was able to show that, when taking a tangent, the curve fragments from the Alberti representations become full Lipschitz curves of $Y$ that are pushed through $f$ to straight lines in $\R^n$.  Using these straight lines, it is possible to show that $f$ is locally surjective in a scale invariant way.

As we cannot take a tangent, we do not have access to these full straight lines, but rather fragmented biLipschitz curves.  The heart of our argument is showing that if a ball $B \subseteq X$ has a high density of points that lie on dense long fragmented curves with some certain speed relative to $\phi$, then the image of $B$ under $\phi$ contains most of a uniformly large ball.  This will be sufficient to use the biLipschitz pieces decomposition of \cite{david-semmes,semmes}.

%In addition, we show that the machinery of finding big biLipschitz pieces from David's condition is applicable when we only have control over the infinitesimal density of the measure and only for a subset of the space, rather than full Ahlfors regularity.

Finally, we note that the decomposition of a rectifiable metric measure space into Lipschitz differentiability spaces, rather than into a single Lipschitz differentiability space with many charts, is necessary.  This is because of the subtle point that a countable union of Lipschitz differentiability spaces need {\it not} actually be a Lipschitz differentiability space, as the pieces may interact in undesirable ways.  This is relevant since, in the definition of the derivative, we require convergence over the whole space.  Of course, if the convergence were only over the set defining a chart, then the union would trivially be a Lipschitz differentiability space.

As an example, consider the subset of the plane defined by
\[(\{0\}\times [0,1]) \cup \{(x,p/2^n): n\in \mathbb N,\ 1\leq p< 2^n \text{ odd},\ \pm x \in [2^{-n}-4^{-n},2^{-n}]\},\]
equipped with $\Hd^1$.  This is a compact metric measure space of finite measure that is a countable union of Lipschitz differentiability spaces and is rectifiable.  However, the function $|x|$ is not differentiable at any point of the vertical segment.  More generally, the functions $|x|$ and $x$ cannot both be differentiable with respect to the same chart function and so the space is not a 1-dimensional Lipschitz differentiability space.  It cannot be a higher dimensional Lipschitz differentiability space for several reasons.  For example, it does not have 2 independent Alberti representations or the fact that the vertical segment is a porous set of positive measure.  In fact, it is easy to see that a countable union of Lipschitz differentiability spaces is a Lipschitz differentiability space if and only if every porous set has measure zero.  Finally, if we restrict the measure to the vertical segment, then we see that it is also necessary to consider only the support of $\mu$, rather than the whole space, in our theorem.\\

\noindent {\bf Acknowledgements.} We would like to thank Marianna Cs\"ornyei for many helpful conversations and Matthew Badger for comments on our preprint.  S. Li was supported by a postdoctoral research fellowship NSF DMS-1303910.

\section{Preliminaries}\label{s:preliminaries}
In this section, we introduce various properties of our metric space.  Unless specified otherwise, we will let $B(x,r)$ denote the closed ball centred at $x$ with radius $r$ in $X$.  For simplicity, we will consider a metric measure space to be a complete and separable metric space equipped with a finite Borel regular measure.  However, our main theorem immediately generalises to any metric measure space with a Radon measure.

Our first goal is to show the existence of a set of dyadic cubes that we describe now.  Unlike previous versions of dyadic cubes, we do not require the full Ahlfors regularity of $X$, but just Ahlfors regular estimates for $X$ around a certain set $K$.  The tradeoff is that we will only guarantee good covering properties for $K$ and we will only get cubes for small enough scales, but this will be sufficient for our purposes.  The proof is similar to that of Theorem 11 of \cite{christ} with some modification.  We will give a proof in Appendix A.

\begin{proposition} \label{p:cubes-1}
  Let $K \subset X$ be a compact subset of a metric measure space $(X,d,\mu)$ that satisfies
  \begin{align}
    \frac{1}{C} r^n \leq \mu(B(x,r)) \leq Cr^n, \qquad \forall x \in K, r < R, \label{e:ball-growth-hyp}
  \end{align}
  for some $R > 0$ and $C > 1$.  Then there exist constants $a > 0$, $\eta > 0$, $k_K \in \Z$ and a collection of subsets $\Delta = \{Q_\omega^k \subset X : k \leq k_K, \omega \in I_k\}$ so that $16^{k_K+2} \leq R$,
  \begin{enumerate}
    \item $\mu\left(K \backslash \bigcup_\omega Q_\omega^k \right) = 0 \qquad \forall k \leq k_K$,
    \item If $j \geq k$, then either $Q_\alpha^k \subset Q_\omega^j$ or $Q_\alpha^k \cap Q_\omega^j = \emptyset$,
    \item For each $(j,\alpha)$ and each $k \in \{j,j+1,...,k_K\}$, there exists a unique $\omega$ so that $Q_\alpha^j \subseteq Q_\omega^k$,
    \item For each $Q_\omega^k$, there is some $z_\omega \in K$ so that
    \begin{align}
      B(z_\omega, 16^{k-1}) \subseteq Q_\omega^k \subseteq B(z_\omega, 16^{k+1}). \label{e:cube-diam}
    \end{align}
    \item For each $k,\alpha$ and $t > 0$, we have
    \begin{align}
      \mu\{x \in Q_\alpha^k \cap K : \dist(x,X \backslash Q_\alpha^k) \leq t 16^k \} \leq at^\eta \mu(Q_\alpha^k). \label{e:small-boundaries}
    \end{align}
    \item For each $k,\alpha$,
    \begin{align}
      \frac{1}{C} 16^{(k-1)n} \leq \mu(Q_\alpha^k) \leq C16^{(k+1)n}. \label{e:cube-growth}
    \end{align}
  \end{enumerate}
\end{proposition}

By Properties 2 and 6 of Proposition \ref{p:cubes-1}, we see that for each $Q_\alpha^k \in \Delta$, the number of cubes $Q_\omega^{k-1}$ that $Q_\alpha^k$ can contain is bounded by a number depending only on $C$ and $n$.  We let $\Delta_k = \{Q_\omega^k \in \Delta : \omega \in I_k\}$.  Given a cube $Q_0$, we let $\Delta(Q_0) = \{Q \in \Delta : Q \subseteq Q_0\}$.  A similar definition gives us $\Delta_k(Q_0)$.  Note that if $k$ is too big, then $\Delta_k(Q_0)$ can be empty.  For a cube $Q$, we let $z_Q$ denote the point guaranteed to us to satisfy \eqref{e:cube-diam}.  We let $j(Q)$ denote the largest integer such that $Q \in \Delta_{j(Q)}$.  For convenience, we will also set $\ell(Q) = 16^{j(Q)}$.

A set of cubes $\calE$ is said to be a Carleson set if there exists some constant $C > 0$ so that the following Carleson estimate is satisfied:
\begin{align*}
  \sum_{Q \in \calE, Q \subseteq Q_0} \mu(Q) \leq C \mu(Q_0), \qquad \forall Q_0 \in \Delta.
\end{align*}
Here, $C$ is called the Carleson constant.  Given some set $A \subseteq X$, we say that $\calE$ is an $A$-Carleson set if
\begin{align*}
  \sum_{Q \in \calE, Q \subseteq Q_0} \mu(Q \cap A) \leq C \mu(Q_0), \qquad \forall Q_0 \in \Delta.
\end{align*}
Obviously, any subset of a ($A$-)Carleson set is also a ($A$-)Carleson set and any set of disjoint cubes is a ($A$-)Carleson set with constant 1.  Also, any Carleson set is an $A$-Carleson set for any $A \subseteq X$.

Given some $\lambda > 1$ and some cube $Q \in \Delta$, we let
\begin{align*}
  \lambda Q = \{ x \in X : \dist(x,Q) \leq (\lambda - 1) \diam(Q) \}.
\end{align*}

From now on, we will assume that $(X,d,\mu)$ is a metric measure space with
$$0< \Theta^n_*(\mu;x)\leq \Theta^{*,n}(\mu;x) < \infty, \qquad \mu\mbox{--}a.e. ~x \in X,$$
and $(U, \varphi)$ is a chart of dimension $n$ such that $\mu \llcorner U$ has $n$ $\phi$-independent Alberti representations and $\mu(U)>0$.  Note that the density conditions imply that $X$ is pointwise doubling and so satisfies the Lebesgue density theorem.  Note also that, by the Lebesgue density theorem, the upper and lower densities of $\mu\llcorner U$, for measurable $U\subset X$, equal the upper and lower densities of $\mu$ at almost every point of $U$. This allows us to freely restrict our attention to measurable subsets of $X$.

Let
\begin{align*}
  U(j,R) = \left\{x \in U : \frac{1}{j}r^n \leq \mu(B(x,r)) \leq jr^n, ~\forall r < R\right\}.
\end{align*}
It is easy to see that each $U(j,R)$ is a Borel set and since $X$ has positive and finite lower and upper $n$-densities almost everywhere, we have that
\begin{align*}
  \mu\left( U \backslash \bigcup_{j=1}^\infty \bigcup_{k=1}^\infty U(j,k) \right) = 0.
\end{align*}
Since a countable union of rectifiable metric measure spaces is rectifiable, it therefore suffices to show that each $U(j,R)$ is rectifiable.  Therefore, by inner regularity of $\mu$, we may reduce proving the rectifiability of $(U,d,\mu)$ to proving the rectifiability of $(K,d,\mu)$, where $K$ is any compact set for which there exist some constants $C > 1$ and $R_K > 0$, which we now fix, so that
\begin{align}
  \frac{1}{C} r^n \leq \mu(B(x,r)) \leq C r^n, \qquad \forall x \in K, r < R_K. \label{e:ball-growth}
\end{align}
An immediate consequence of these bounds is
\begin{align}
  \frac{1}{2^n C} \Hd^n \llcorner K \leq \mu \llcorner K \leq 2^nC \Hd^n \llcorner K. \label{e:mu-comparison}
\end{align}
Let $\Delta$ denote the collection of cubes covering $K$ as given by Proposition \ref{p:cubes-1}.

We now discuss the Alberti representations in $X$.  Firstly, it is easy to see that the set of $L$-Lipschitz curves in $\Gamma$ is a closed subset of $\Gamma$ (see \cite[Lemma 2.7]{bate} and also for the topology used in $\Gamma$).  Therefore we may rescale the domains of the curves in $\Gamma$ using a Borel function, so that we may suppose $\Gamma$ consists of only 1-Lipschitz functions.

Suppose that we have an Alberti representation in the $\varphi$-direction of a cone $C$ and that $C_1,\ldots,C_k$ are cones that cover an open neighbourhood $N$ of $C$.  Then it is possible to find a Borel decomposition $X=X_1 \cup \ldots \cup X_k$ such that each $\mu\llcorner X_i$ has an Alberti representation in the $\phi$-direction of $C_i$.  This process is known as {\it refining} the Alberti representations, see \cite[Definition 5.10]{bate}.  Therefore, we may suppose that the Alberti representations are in the $\phi$-direction of cones of width $1/1000n^2$.  Further, by applying a linear transformation to $\phi$, we may suppose that these Alberti representations are in the $\phi$-direction of cones centred on the standard basis vectors.  Of course, this linear transformation may not be angle preserving.  However, the amount that the angles may increase is determined by the cones that we started with, before any refining.  Therefore, by making the initial refinement sufficiently fine, we may suppose that the conditions on the centres and the widths of the refined cones are both satisfied.

Finally, we may suppose that $\phi$ is 1-Lipschitz.

For $R,v>0$ we define the set $GP(v, R)$ to be those $y \in K$ for which, for each $1 \leq j \leq n$, there exists a $\gamma \in \Gamma$ and a $t$ in the domain of $\gamma$ such that
\begin{enumerate}
\item $\gamma(t)= y$;
\item for every $0< r < 4R\operatorname{biLip}(\gamma)$, $|B(t,r) \cap \gamma^{-1}(K)| > (1-1/100000n)|B(t,r)|$ \label{e:gamma-density};
\item for every $s>s' \in \dom\gamma$, $\phi(\gamma(s))-\phi(\gamma(s')) \in C(e_j, 1/1000n^2)$ and
\begin{align}
  |\phi(\gamma(s))-\phi(\gamma(s'))| > v d(\gamma(s), \gamma(s')). \label{e:phi-speed}
\end{align}
\end{enumerate}
Since $X$ is complete and separable and $\mu$ is Borel regular, the results of \cite[Section 2]{bate} show that $GP(v,R)$ is measurable for any $v, R>0$.  We remark that the final condition in the definition of $GP(v,R)$ is related to the speed of an Alberti representation: an Alberti representation has $\phi$-speed $\delta$, for $\delta>0$, if
\[(\varphi\circ\gamma)'(t)> \delta \Lip(\phi,\gamma(t))\Lip(\gamma,t)\]
for $\mathbb P$ a.e. $\gamma\in \Gamma$ and $\mu_\gamma$ a.e. $t\in \dom\gamma$.  Although our Alberti representations may not have a speed, it is easy to see that the space may be decomposed into pieces on which the induced Alberti representations do have a certain speed (see \cite[Corollary 5.9]{bate}).  Further, \cite[Proposition 2.9]{bate} shows that $GP(v, R)$ converges to a set of full measure in $K$ as $v,R \to 0$.

We also define
\begin{align*}
  DP(v,\epsilon,R) = \{x \in K : \mu(GP(v,R) \cap B(x,r) )\geq (1-\epsilon) \mu(B(x,r)), \forall r < R \}.
\end{align*}
and
\begin{align*}
  DC(\beta,\epsilon,R) = \{x \in K : |B(\phi(x),\beta r) \cap \phi(B(x,r) \cap K)| \geq (1-\epsilon) |B(\phi(x),\beta r)|, \forall r < R\}.
\end{align*}
For convenience, we let $|\cdot|$ denote the Hausdorff $n$-measure on $\R^n$ rather than the Lebesgue measure, although the definition of $DC$ is clearly invariant under scalings of the measure.  Note that $DC(\beta,\epsilon,R) \subseteq DC(\beta',\epsilon,R)$ and $DC(\beta,\epsilon,R) \subseteq DC(\beta,\epsilon,R')$ when $\beta' \leq \beta$ and $R' \leq R$.

The points of $DC$ satisfy one of the two conditions of David's condition.  This is the more important of the two condition that allows one to deduce biLipschitz behavior.  See Section 1 and in particular equation (9) in \cite{david} or Section 9 and in particular Condition 9.1 and Remark 9.6 of \cite{semmes} for more information on David's condition as it was originally introduced.

We can now specify what we mean in the statement of condition \eqref{i:main-david} in Theorem \ref{t:main}.  We say that $U$ satisfies David's condition a.e. with respect to the function $\phi : U \to \R^n$ if for every $\epsilon \in (0,1)$,
\begin{align*}
  \mu\left(U \backslash \bigcup_{j=1}^\infty \bigcup_{k=1}^\infty DC(1/j,\epsilon,1/k) \right) = 0,
\end{align*}
where the $DC$ sets here are defined with respect to $U$ and $\phi$.

\section{David's condition}

The goal of this section is to prove the following proposition, which proves that Condition \eqref{i:main-alberti} implies Condition \eqref{i:main-david} of Theorem \ref{t:main}.

\begin{proposition} \label{p:dc-decomposition}
  For all $\epsilon \in (0,1)$,
  $$\mu \left( K \backslash \bigcup_{j=1}^\infty \bigcup_{k=1}^\infty DC( 1/j, \epsilon, 1/k) \right) = 0.$$
\end{proposition}

We will need the following lemma, which is the heart of the iterative step needed for our grid construction.

\begin{lemma} \label{l:induction-alternative}
  Let $x \in K$, $v>0$ and $R > r > 0$.  Let $Q$ be an axis-parallel cube in $\R^n$ that has sidelength $\frac{v}{10n}r$ and whose center $x_Q$ satisfies $|\phi(x) - x_Q| < \frac{v}{100n}r$.  Let $\{p_i\}_{i=1}^{2^n}$ denote the centers of the $2^n$ quadrant subcubes of $Q$.  Then at least one of the following must be true:
  \begin{enumerate}
    \item There exists points in $\{q_i\}_{i=1}^{2^n}$ in $B(x,r/2) \cap K$ so that $|\phi(q_i) - p_i| < \frac{v}{1000n}r$.
    \item There exists some $y \in B(x,r/2) \cap K$ so that $B\left(y, \frac{v}{10000n^3}r\right) \cap GP(v, R) = \emptyset$ and $\dist(\phi(y),Q^c) \geq \frac{v}{100n}r$.
  \end{enumerate}
\end{lemma}

\begin{proof}
  Suppose the second alternative is false.  We may suppose without loss of generality that $x_Q = 0$.  Let us first assume that $n=2$.  Note then that the centers of the 4 quadrant cubes are located at
  \begin{align*}
    \left( \pm \frac{v}{40n}r, \pm \frac{v}{40n}r \right), \left( \pm \frac{v}{40n}r, \mp \frac{v}{40n}r \right).
  \end{align*}
  Consider the set $B\left(x, \frac{v}{10000n^3}r\right) \cap G(v, R)$.  As $|\phi(x) - x_Q| < \frac{v}{100n}r$, the fact that the second alternative is false means that this set is not empty and so there exist $\gamma_0 : D_0 \to K$ and $t_0 \in \dom\gamma$ that satisfy the defining properties of $G(v, R)$ for the cone $C(e_1, 1/1000n^2)$ such that $\gamma_0(t_0) \in B\left(x, \frac{v}{10000n^3} r\right) \cap G(v, R)$.  By Property \ref{e:gamma-density} of $GP(v,R)$, we have that there exist points $t_1,t_2$ such that
  \begin{align}
    \left|\phi_1(\gamma_0(t_1)) + \frac{v}{40n}r \right| < \frac{v}{10000n^2}r, \qquad \left| \phi_1(\gamma_0(t_2)) - \frac{v}{40n}r \right| < \frac{v}{10000n^2}r. \label{eq:x-close}
  \end{align}
  Now consider all points of $B\left(\gamma_0(t_1), \frac{v}{10000n^3}r\right) \cap G(v, R)$.  By \eqref{e:phi-speed} in the definition of $G(v, R)$, we have
  \begin{align*}
    d(\gamma_0(t_1),x) \leq d(\gamma_0(t_1),\gamma_0(t_0)) + d(\gamma_0(t_0),x) \leq \frac{1}{10n}r + \frac{v}{10000n^3}r \leq \frac{1}{2} r.
  \end{align*}
  It is easy to see that $\dist(\phi(\gamma_0(t_1)),Q^c) \geq \frac{v}{100n}r$.  Thus, as the second alternative is false, we can take such a curve $\gamma_1 : D_1 \to K$ for the cone $C(e_2,1/1000n^2)$ and let $\gamma_1(t_{10}) \in B\left(\gamma_0(t_1), \frac{v}{10000n^3}r \right) \cap K$.  Again, by Property \ref{e:gamma-density} of $GP(v,R)$, there exist points $t_{11}$ and $t_{12}$ such that
  \begin{align*}
    \left|\phi_2(\gamma_1(t_{11})) + \frac{v}{40n}r \right| < \frac{v}{10000n^2}r, \qquad \left| \phi_2(\gamma_1(t_{12})) - \frac{v}{40n}r \right| < \frac{v}{10000n^2}r.
  \end{align*}
  As $\gamma_1$ is traveling in the $\phi$-cone $C\left(e_2,\frac{1}{1000n^2}\right)$ and starts off at a point near $\gamma_0(t_1)$, we get that
  \begin{align*}
    \left|\phi_1(\gamma_1(t_{11})) + \frac{v}{40n}r \right| &\leq \left| \phi_1(\gamma_1(t_{11})) - \phi_1(\gamma_1(t_{10})) \right| + \left| \phi_1(\gamma_1(t_{10})) - \phi_1(\gamma_0(t_1)) \right| \\
    &\qquad + \left| \phi_1(\gamma_0(t_1)) + \frac{v}{40n}r \right| \\
    &\overset{\eqref{eq:x-close}}{\leq} \frac{v}{10000n^3} r + \frac{v}{10000n^3} r + \frac{v}{10000n^2} r \leq \frac{v}{1000n^2}r.
  \end{align*}
  The first term on the right hand side above comes from the fact that the curve is traveling in a cone of width $\frac{1}{1000n^2}$ in the direction of $e_2$ over a length of no more than $\frac{v}{10n}r$.  The second term comes from how we chose $\gamma_1(t_{10})$ and the third term comes from \eqref{eq:x-close}.  A similar bound holds for $\left| \phi_1(\gamma_1(t_{12})) + \frac{v}{40n}r \right|$.
  
  Thus, we see that $\phi(\gamma_1(t_{11}))$ and $\phi(\gamma_1(t_{12}))$ are within $\frac{v}{1000n}r$ of the centers of the left two quadrant subcubes of $Q$.  One may do a similar construction starting from $\gamma_0(t_2)$ to get a curve $\gamma_2$ and points $\phi(\gamma_2(t_{21}))$ and $\phi(\gamma_2(t_{22}))$ that are within $\frac{v}{1000n}r$ of the centers of the right two quadrant subcubes of $Q$.

  In the case of general $n$, one continues using curves in the remaining directions in the obvious way.  That the width of the aperature of the cones is $O(n^{-2})$ and the radius of the balls in which we are looking for curves is $O(n^{-3})$ allows us to guarantee that the errors accumulated in finding new curves is controlled.

  Consider one of these points $p$ that we have found near the center of a quadrant subcube.  We will again assume $n = 2$ for simplicity.  Let $p = \phi(\gamma_1(t_{11}))$, for instance.  Remember that $\gamma_0$ and $\gamma_1$ both satisfy \eqref{e:phi-speed}, we have that $d(\gamma_i(s),\gamma_i(t)) < \frac{1}{v} |\phi(\gamma_i(s)) - \phi(\gamma_i(t))|$.  Thus, we have that
  \begin{align*}
    d(\gamma_1(t_{11}),x) &\leq d(\gamma_1(t_{11}),\gamma_1(t_{10})) + d(\gamma_1(t_{10}),\gamma_0(t_1)) + d(\gamma_0(t_1),\gamma_0(t_0)) + d(\gamma_0(t_0),x) \\
    &\leq \frac{1}{10n}r + \frac{v}{10000n^3}r + \frac{1}{10n}r + \frac{v}{10000n^3}r \leq \frac{1}{2}r.
  \end{align*}
  A similar estimate holds for the other quadrants.  The case of general $n$ follows completely analogously.  Thus, we have shown in the case that the second alternative of the lemma is false that the first alternative must be true.
\end{proof}

We can now prove that points whose neighborhoods have a high density of points in $GP$ belong to some $DC$ set.

\begin{lemma} \label{l:david-inclusion}
  There exists some constant $\alpha > 0$ depending only on $K$ so that for any $\epsilon,v,R > 0$ such that $R < R_K$,
  \begin{align*}
    DP(v,\epsilon, R) \subseteq DC\left( \frac{v}{20n},\alpha \epsilon,R\right).
  \end{align*}
\end{lemma}

\begin{proof}
  Let $x \in DP(v,\epsilon, R)$ and let $r < R$.  Since $K$ is compact, so is $\phi(B(x,r) \cap K)$.  Let $Q$ be the axis-parallel cube centered at $\phi(x)$ with sidelength $\frac{v}{10n}r$.  We will show for some constant depending only on $K$ that
  \begin{align*}
    |Q \backslash \phi(B(x,r) \cap K)| \leq C\epsilon \mu(B(x,r)).
  \end{align*}
  By \eqref{e:ball-growth}, this clearly will suffice.

  We will define the following stopping time process.  In the first stage, we start off with $(x,Q,r)$.  If the first alternative of Lemma \ref{l:induction-alternative} is satisfied with this triple, we can apply it to get $\{q_i\}_{i=1}^{2^n} \subset K$, points which map close to the center of the quadrant subcubes $\{Q_i\}_{i=1}^{2^n}$ of $Q$.  Otherwise, we terminate the process.  Assuming the process continues, for our second stage, we see if the first alternative of Lemma \ref{l:induction-alternative} applies to the triples $\{(q_i,Q_i,r/2)\}_{i=1}^{2^n}$.  Indeed, we can apply the lemma again given the conclusions of the first alternative.  We stop the process at each subcube where the first alternative fails and continue in the cubes where it doesn't, making sure to divide $r$ by a further factor of 2 at each stage.  Note that all the points of $Q$ that the process discovers has a preimage in $B(x,r)$ as the points discovered at each stage are no more than $2^{-k-1}r$ away from a point from the previous stage.

  The cubes where the process terminates after finite time $\{S_i\}_{i=1}^\infty$ are disjoint dyadic subcubes of $Q$.  We will upper bound their collective volume.  If $S_i$ is a cube where the process terminates, then by the failure of the first alternative of Lemma \ref{l:induction-alternative}, the second alternative must be true.  Thus, there exists a ball $B_i \subseteq X$ with center in $K$ of radius comparable to $\ell(S_i)$ completely contained in $B(x,r) \backslash GP(v, R)$.  By \eqref{e:ball-growth}, we then have that $|S_i| \leq C \mu(B_i)$ for some constant $C > 0$ depending only on $K$.  Note that all these balls must be disjoint.  Indeed, let $B_i = B(x_i,\frac{v}{10000n^3}r_i)$ and $B_j = B(x_j, \frac{v}{10000n^3}r_j)$ be two such balls and suppose without loss of generality that $r_i \leq r_j$.  Then by the second alternative, and the fact that $\phi$ is 1-Lipschitz, $d(x_i,x_j) \geq \frac{v}{100n}r_j$.  It easily follows from the triangle inequality that $B_i \cap B_j = \emptyset$.  Thus, we get that
  \begin{align}
    \sum_{i=1}^\infty |S_i| \leq C \sum_{i=1}^\infty \mu(B_i) \leq C \mu(B(x,r) \backslash GP(v,R)) \leq C \epsilon \mu(B(x,r)). \label{e:holes-bound}
  \end{align}
  In the last inequality, we used the fact that $x \in DP(v,\epsilon, R)$.
  
  For almost every $p \in Q \backslash \bigcup_{i=1}^\infty S_i$,
  $$\left|B(p,s) \cap \left(Q \backslash \bigcup S_i\right)\right| > 0, \qquad \forall s > 0.$$
  In particular, the process tells us that there must be some point $y \in B(x,r) \cap K$ that maps into $B(p,s)$.  As this holds true for all $s > 0$, we see that $p$ is a limit point of $\phi(B(x,r) \cap K)$.  Thus, as $\phi(B(x,r) \cap K)$ is compact, we get that almost every point of $Q \backslash \bigcup S_i$ is contained in $\phi(B(x,r) \cap K)$.  We then have that
  \begin{align*}
    \left| Q \backslash \phi(B(x,r) \cap K) \right| \leq \left| \bigcup_{i=1}^\infty S_i \right| \overset{\eqref{e:holes-bound}}{\leq} C\epsilon \mu(B(x,r)).
  \end{align*}
\end{proof}

We are now ready to prove the main proposition of this section.

\begin{proof}[Proof of Proposition \ref{p:dc-decomposition}]
First observe that
\[\phi(B(x,r)\cap K) \subset B(\phi(B(y,r)\cap K),d(y,x))\]
and so, since $\phi(B(x,r)\cap K)$ is closed, $x \mapsto |\phi(B(x,r)\cap K)|$ is upper semicontinuous.  Therefore, since Lebesgue measure is continuous on balls, we see that $DC(\beta,\epsilon, R)$ is closed for any $\beta, \epsilon, R>0$ and hence measurable.

Secondly, we know that $\mu\left(K\setminus \bigcup_j \bigcup_k GP(1/j,1/k)\right)=0$ and so, by the Lebesgue density theorem, for any $\epsilon >0$, $\mu\left(K\setminus \bigcup_j \bigcup_k DP(1/j,\epsilon, 1/k)\right) = 0$.  Therefore, Lemma \ref{l:david-inclusion} concludes the proof.
\end{proof}

\section{BiLipschitz pieces}
The main result of this section is Proposition \ref{p:local-decompose}, which will be the central step in showing that Condition \eqref{i:main-david} implies Condition (R) in Theorem \ref{t:main}.  A good knowledge of \cite{david-semmes,semmes} will be necessary in this section.  Fix a cube $Q_0 \in \Delta$ for the remainder of the section.  We recall some terminology from \cite{semmes}.

Given $\delta > 0$, the set $\SI(\delta,Q_0)$ are the subcubes $Q \in \Delta(Q_0)$ for which there is some $W \in \Delta(Q_0)$ such that $Q \subseteq W$ and $|\phi(W \cap K)| < \delta \mu(W)$.  We may drop the $Q_0$ from the parameters list if it is obvious.

Given $A > 1$, we say that two dyadic cubes are $A$-neighbors (or just neighbors) if
$$\dist(Q,Q') \leq A(\diam Q + \diam Q')$$
and
$$\frac{1}{A} \diam(Q) \leq \diam(Q') \leq A \diam(Q).$$
For some $Q \in \Delta$, we let
\begin{align*}
  \hat{Q} = \left( \bigcup \{S \in \Delta_{j(Q)} : \dist(S,Q) \leq \diam(Q) \} \right) \cap Q_0.
\end{align*}
Given $A > 1$ and $\zeta > 0$, the set $\M_A(\zeta,Q_0)$ are the subcubes $Q \in \Delta(Q_0)$ that satisfy the following property:
\begin{itemize}
  \item $|\phi(Q \cap K)| \geq (1+\zeta)^{-1} \delta \mu(Q)$,
  \item if $R \in \Delta$ is a neighbor of $Q$, then $R \subseteq Q_0$, and
  \begin{align*}
    (1+\zeta)^{-1} \frac{|\phi(Q \cap K)|}{\mu(Q)} \leq \frac{|\phi(R \cap K)|}{\mu(R)} \leq (1+\zeta) \frac{|\phi(Q \cap K)|}{\mu(Q)},
  \end{align*}
  \item if $R \in \Delta$ is a neighbor of $Q$,
  \begin{align*}
    (1+\zeta)^{-1} \frac{|\phi(Q \cap K)|}{\mu(Q)} \leq \frac{|\phi(\hat{R} \cap K)|}{\mu(\hat{R})} \leq (1+\zeta) \frac{|\phi(Q \cap K)|}{\mu(Q)}.
  \end{align*}
\end{itemize}
We will not actually use the first property of $M_A(\zeta,Q_0)$ cubes, but they are defined this way in \cite{semmes} (although without the intersection with $K$) so we keep it to reduce confusion.

Finally, for $\eta > 0$, we define the set $\LD(\eta,Q_0)$ to be the subcubes $Q \in \Delta(Q_0)$ for which there is some $W \in \Delta(Q_0)$ such that $Q \subseteq W$ and $\mu(W \cap K) < \eta \mu(W)$.  Again, we may drop the $Q_0$ from the list of parameters of both $M$ and $\LD$ if it is obvious.

We first prove that the measure of the cubes of $\SI$ can be bounded by the measure of the complement of $DC$.
\begin{lemma} \label{l:Sigma-bound}
  There exists constants $c,C > 0$ depending only on $K$ so that if $\delta > 0$, $\Sigma(c\delta^n) = \bigcup_{Q \in \SI(c\delta^n,Q_0)} Q$ and $\epsilon \in (0,1/2)$, then
  \begin{align}
    \mu(\Sigma(c\delta^n)) \leq C\mu(Q_0 \backslash DC(\delta,\epsilon,16\ell(Q_0))). \label{e:Sigma-bound}
  \end{align}
\end{lemma}

\begin{proof}
  Let $\{Q_i\}_{i=1}^\infty$ denote the set of maximal cubes of $\SI(c\delta^n)$, which are obviously disjoint.  Then $\Sigma(c\delta^n) = \bigcup_{i=1}^\infty Q_i$.  Let $Q \in Q_0$ and suppose $x \in B(z_Q,\ell(Q)/32) \cap DC(\delta,\epsilon,16\ell(Q_0)) \neq \emptyset$.  Then
  \begin{align}
          |\phi(Q \cap K)| \geq |\phi(B(x,\ell(Q)/32) \cap K) \cap B(\phi(x),\delta \ell(Q)/32)| \geq \frac{\delta^n}{2\cdot 32^{n}}  |B(\phi(x),\ell(Q))|.\label{e:qnotinsi}
  \end{align}
  As $B(x,\ell(Q)/32) \subseteq Q$, we get from \eqref{e:cube-diam}, \eqref{e:cube-growth}, and \eqref{e:ball-growth} that $Q \notin \SI(c\delta^n)$ for some small enough $c > 0$ depending only on $K$.  Thus, if $Q \in \SI(c\delta^n)$, then $B(z_Q, \ell(Q)/32) \subseteq Q \backslash DC(\delta,\epsilon, 16\ell(Q_0))$.  As there exists some constant $C > 0$ depending only on $K$ so that $\mu(Q_i) \leq C \mu\left(B(z_{Q_i}, \ell(Q_i)/32)\right)$, we have that
  \begin{align*}
    \mu(\Sigma(c\delta^n)) \leq \sum_{i=1}^\infty \mu(Q_i) \leq C \sum_{i=1}^\infty \mu\left( B(z_{Q_i}, \ell(Q_i)/32) \right) \leq C \mu(Q_0 \backslash DC(\delta,\epsilon, 16\ell(Q_0))).
  \end{align*}
  Note that, for the final inequality, we have used the fact that the balls $B(z_{Q_i}, \ell(Q_i)/32)$ are disjoint.
\end{proof}

We can also show that the space covered by $\LD$ has small volume.
\begin{lemma} \label{l:Lambda-bound}
  Let $\eta > 0$ and $\Lambda(\eta) = K \cap \bigcup_{Q \in \LD(\eta,Q_0)} Q$.  Then
  \begin{align*}
    \mu(\Lambda(\eta)) < \eta \mu(Q_0).
  \end{align*}
\end{lemma}

\begin{proof}
  Let $\{Q_i\}_{i=1}^\infty$ denote the set of maximal cubes of $\LD(\eta)$, which are obviously disjoint.  Then $\Lambda(\eta) = \bigcup_{i=1}^\infty Q_i$ and so
  \begin{align*}
    \mu(\Lambda(\eta)) \leq \sum_{i=1}^\infty \mu(Q_i \cap K) < \eta \sum_{i=1}^\infty \mu(Q) \leq \eta \mu(Q_0).
  \end{align*}
\end{proof}

By Lemma 8.2 of \cite{semmes} (or one can easily derive from \eqref{e:cube-diam}), we can choose some absolute constant $b \in (0,1)$ so that if $x,y \in K$ are distinct points and $Q$ is the smallest cube in $\Delta$ such that $x \in Q$ and $y \in 2Q$, then
\begin{align*}
  d(x,y) \geq 10b \diam(Q).
\end{align*}

\begin{lemma} \label{l:weak-bilipschitz}
  There exist constants $k > 0$ depending only on $\beta$ and $\zeta_0,A_0 > 0$ depending on $k$ and $K$ so that the following holds.  Let $\zeta < \zeta_0$, $A > A_0$, $\epsilon \in (0,1/10)$, $Q \in M_A(\zeta,Q_0)$.  If $x,y \in 2Q \cap Q_0 \cap DC(\beta,\epsilon, 16\ell(Q_0))$ are such that $d(x,y) > b \diam(Q)$, then
  \begin{align*}
    |\phi(x) - \phi(y)| \geq k^{-1} d(x,y).
  \end{align*}
\end{lemma}

\begin{proof}
  Given a $j \in \Z$, we define
  \begin{align*}
    T_j(x) = \bigcup \{Q \in \Delta_j : Q \cap B(x,16^j) \neq \emptyset\}.
  \end{align*}
  It is clear then that $B(x,16^j) \subseteq T_j(x) \subseteq B(x,16^{j+2})$.

  The proof follows the proof of Proposition 9.36 of \cite{semmes}.  Let us suppose that the conclusion does not hold, so that $|\phi(x)-\phi(y)|< d(x,y)/k$ for some $k > 0$ to be determined, and seek a contradiction.  Let $j_1$ be the largest integer at most $j(Q)$ so that
  \begin{align*}
    T_{j_1}(x) \cap T_{j_1}(y) = \emptyset.
  \end{align*}
  As $d(x,y) > b \diam(Q)$, we get that there exists some absolute constant $C > 0$ so that
  \begin{align}
    0 \leq j(Q) - j_1 \leq C. \label{e:j1-big}
  \end{align}
  We will show there exists some $c > 0$ and $k > 0$ depending only on $\beta$ so that
  \begin{align}
    |\phi(B(x,16^{j_1}) \cap K) \cap \phi(B(y,16^{j_1}) \cap K)| \geq c |\phi(Q \cap K)|. \label{e:ball-overlap}
  \end{align}
  This then proves that
  \begin{align}
    |\phi(T_{j_1}(x) \cap K) \cap \phi(T_{j_1}(y) \cap K)| \geq c |\phi(Q \cap K)|, \label{e:large-overlap}
  \end{align}
  which is equation (9.50) of \cite{semmes}.  The rest of the proof will just continue as in the proof of Proposition 9.36 after Remark 9.56 of \cite{semmes}.  We quickly go over the idea of the argument although we leave the details to the reader.
  
  By choosing $A$ large enough, we can find a neighbor $Q_1$ of $Q$ such that $\widehat{Q}_1$ contains both $T_{j_1}(x)$ and $T_{j_1}(y)$.  By letting $A$ be sufficiently large again, we can also ensure that the cubes of $\Delta_{j_1}$ that make up $\widehat{Q}_1$ are all neighbors of $Q$.  Thus, as $Q \in M_A(\zeta)$, we have that each of these cubes $W \in \Delta_{j_1}$ so that $W \subset \widehat{Q}_1$ satisfies
  \begin{align}
    \frac{|\phi(W \cap K)|}{\mu(W)} \leq (1 + \zeta) \frac{|\phi(Q \cap K)|}{\mu(Q)}. \label{e:W-preserve}
  \end{align}
  We also have that 
  \begin{align}
    (1+\zeta)^{-1} \frac{|\phi(Q \cap K)|}{\mu(Q)} \leq \frac{|\phi(\widehat{Q}_1 \cap K)|}{\mu(\widehat{Q}_1)}. \label{e:Q1-preserve}
  \end{align}
  If one chooses $\zeta$ small enough, one can get that \eqref{e:Q1-preserve} is not consistent with \eqref{e:large-overlap} and \eqref{e:W-preserve}.  This is because \eqref{e:large-overlap} says that the images of $T_{j_1}(x)$ and $T_{j_1}(y)$ overlap and eat up too much of each other's measure and \eqref{e:W-preserve} says that the other cubes of $\hat{Q}_1$ cannot make up for this lost measure.  This will ensure that we cannot satisfy \eqref{e:Q1-preserve}.  We thus have a contradiction and so our initial assumption that $|\phi(x) - \phi(y)| < k^{-1} d(x,y)$ was false.  The reader can consult the proof of Proposition 9.36 in \cite{semmes} for full details.

  We now prove \eqref{e:ball-overlap}.  By Lemma 9.10 of \cite{semmes}, if we choose $A$ to be large enough, then $T_{j_1}(x)$ and $T_{j_1}(y) \subseteq Q$.  Thus, as $\epsilon \in (0,1/4)$ and $x,y \in DC(\beta,\epsilon, 16\ell(Q_0))$, we get that
  \begin{align*}
    |B(\phi(x),\beta 16^{j_1}) \cap \phi(B(x, 16^{j_1}) \cap K)| &\geq \frac{3}{4} |B(\phi(x),\beta 16^{j_1})|,\\
    |B(\phi(y),\beta 16^{j_1}) \cap \phi(B(y, 16^{j_1}) \cap K)| &\geq \frac{3}{4} |B(\phi(y),\beta 16^{j_1})|.
  \end{align*}
  Note that
  \begin{align*}
    d(x,y) \leq 3 \diam(Q) \overset{\eqref{e:cube-diam} \wedge \eqref{e:j1-big}}{\leq} 96 \cdot 16^C 16^{j_1}.
  \end{align*}
  Thus, if $k$ is larger than some constant depending only on $\beta$, then
  \begin{align*}
    |\phi(B(x,16^{j_1}) \cap K) \cap \phi(B(y, 16^{j_1}) \cap K)| \geq \frac{1}{4} |B(\phi(x),\beta 16^{j_1})| = (*).
  \end{align*}
  By \eqref{e:cube-diam}, $Q$ has diameter no more than $32 \cdot 16^{j(Q)} \leq 32 \cdot 16^C 16^{j_1}$.  Thus, $Q \subseteq B(x,3\diam(Q)) \subseteq B(x, 96 \cdot 16^C 16^{j_1})$.  As $\phi$ is 1-Lipschitz, we have that
  \begin{align*}
    (*) \geq \frac{\beta^n}{384^n 16^{Cn}} |B(\phi(x), 96 \cdot 16^C 16^{j_1})| \geq \frac{\beta^n}{384^n 16^{Cn}} |\phi(Q \cap K)|.
  \end{align*}
  This finishes the proof of \eqref{e:ball-overlap}.
\end{proof}

The following proposition says that the set of cubes not in $\SI(\delta)$, $\LD(\eta)$, and $M_A(\zeta,Q_0)$ satisfy a $K$-Carleson estimate.  It is mostly proven in \cite{semmes}, but we will require some nontrivial changes.  The proof will be given in Appendix B.

\begin{proposition} \label{p:M-carleson}
  Let $\eta,\delta > 0$ and $A > 1$.  Then there exists some $\zeta_0 > 0$ so that for each $\zeta < \zeta_0$ there exists some $C > 0$ depending only these constants and $K$ so that
  \begin{align*}
    \sum_{Q \in \Delta(Q_0) \backslash (M_A(\zeta,Q_0) \cup \SI(\delta,Q_0) \cup \LD(\eta,Q_0))} \mu(Q \cap K) \leq C \mu(Q_0).
  \end{align*}
\end{proposition}

We can now prove the main result of this section.  We are not keeping track of the number of biLipschitz pieces or the biLipschitz constant---although they can be estimated---as they are not necessary for our purposes.

\begin{proposition} \label{p:local-decompose}
  Let $\eta,\delta,\epsilon > 0$, and $Q_0 \in \Delta$ and let $\phi : K \to \R^n$ be a chart.  Assume $(Q_0 \cap K) \backslash (\Sigma(c\delta^n) \cup \Lambda(\eta))$ has positive measure where $c > 0$ is the constant from Lemma \ref{l:Sigma-bound}.  Then there exists some $C > 0$ depending only on $\delta$, $\eta$, and $K$ so that the following holds.  There exist a finite collection of compact sets $\{F_j\}_{j=1}^M$ in $Q_0 \cap K$ such that $\phi|_{F_j}$ is biLipschitz and
  \begin{align*}
    \mu\left((Q_0 \cap K) \backslash \bigcup F_j\right) \leq \epsilon + \eta \mu(Q_0) + C \mu(Q_0 \backslash DC(\delta, 1/10, 16\ell(Q_0))).
  \end{align*}
\end{proposition}

\begin{proof}
  The proof will follow the usual biLipschitz decomposition method of \cite{david-semmes,jones}, except we now use Lemma \ref{l:Sigma-bound} and Lemma \ref{l:Lambda-bound} to control the measure of $\Sigma(c\delta^n)$ and $\Lambda(\eta)$ and we localize to a $DC$ set.  We proceed.

  We choose $A,\zeta,k$ depending only on $\delta$ and $K$ so that Lemma \ref{l:weak-bilipschitz} applies to points of $DC(\delta,1/4, 16\ell(Q))$ in the cubes of $M_A(\zeta,Q_0)$ and so that $\zeta$ is also small enough to apply Proposition \ref{p:M-carleson}.  For a cube $Q$, let $\tilde{Q}$ denote the union of cubes in $Q' \in \Delta_{j(Q)}(Q_0)$ so that $Q' \cap 2Q \neq \emptyset$.  For $L \geq 1$, we define the set
  \begin{align*}
    R_L = \left\{ x \in (Q_0 \cap K) \backslash (\Sigma(c\delta^n) \cup \Lambda(\eta)) : \sum_{Q \in \Delta(Q_0) \backslash M_A(\zeta)} \chi_{\tilde{Q}}(x) \geq L \right\}.
  \end{align*}
  As $(Q_0 \cap K) \backslash (\Sigma(c\delta^n) \cap \Lambda(\eta))$ has positive measure, by Proposition \ref{p:M-carleson}, there exists some $L' > 0$ depending only on $A,\zeta,\delta,\eta$ so that
  \begin{align}
    \mu(R_{L'}) < \epsilon. \label{e:small-RL}
  \end{align}
  The usual coding argument from \cite{jones} or Section 2 of \cite{david-semmes} required to decompose the set
  $$(Q_0 \cap DC(\delta,1/10, 16\ell(Q))) \backslash (\Sigma(c\delta^n) \cup \Lambda(\eta) \cup R_{L'})$$
  into the pieces on which $\phi$ is biLipschitz remains unchanged.  Indeed, the fact that we take points from $DC(\delta,1/10, 16\ell(Q))$ allows us to get a weak $k$-biLipschitz behavior from the cubes of $M_A(\zeta)$ by Lemma \ref{l:weak-bilipschitz}.  For the convenience of hte reader, we provide a quick sketch of the argument at the end of this proof.

  We now bound the measure of $(Q_0 \cap K) \backslash \bigcup F_j$.  We have that
  \begin{align*}
    \mu((Q_0 \cap K) \backslash \bigcup F_j) &\leq \mu(R_{L'}) + \mu(\Lambda(\eta)) + \mu(\Sigma(c\delta^n)) + \mu((Q_0 \cap K) \backslash DC(\delta,1/10, 16\ell(Q_0))) \\
    &\overset{\eqref{e:Sigma-bound} \wedge \eqref{e:small-RL}}{\leq} \epsilon + \eta \mu(Q_0) + C \mu(Q_0 \backslash DC(\delta,1/10, 16\ell(Q_0))) \\
    &\qquad + \mu(Q_0 \backslash DC(\delta,1/10, 16\ell(Q_0))) \\
    &\leq \epsilon + \eta \mu(Q_0) + (C+1) \mu(Q_0 \backslash DC(\delta,1/10, 16\ell(Q_0))).
  \end{align*}
  The dependencies for $C$ are $\zeta,A,\delta,\eta$, but $\zeta$ and $A$ depend on $\delta,\eta$ and $K$ so $C$ depends only on $\delta$, $\eta$, and $K$.

  The only thing left is to give a sketch of the encoding argument for biLipschitz decomposition.  Let $\ell$ be a constant large enough so that if $S \in \Delta_{k+\ell}$ and $Q \in \Delta_k$, then $\diam Q < b \diam S$.  For each $k$ and $Q \in \Delta_k(Q_0)$, we let $\calF(Q)$ denote the set of cubes $Q' \in \Delta_k(Q_0)$ so that $Q' \neq Q$ and $Q,Q' \subset \widetilde{S}$ for some $S \in \Delta_{k+\ell}(Q_0) \backslash M_A(\zeta)$.  We get from the volume estimates of the cubes that $\calF(Q)$ can contain at most $T$ elements where $T$ depends only on $K$.

  We let $A$ be a set of $T+1$ distinct characters.  We will associate to each $Q \in \Delta(Q_0)$ a (possibly empty) string of these characters $a(Q)$, which we will call words.  For $Q \in \Delta$, let $Q^*$ denote the unique parent of $Q$.  The words will satisfy the following properties: $a(Q_0) = \emptyset$, if $\calF(Q) = \emptyset$ then $a(Q) = a(Q^*)$, if $\calF(Q) \neq \emptyset$ then $a(Q)$ is constructed by appending a letter on the end of $a(Q^*)$ so that if $Q' \in \calF(Q)$ then
  \begin{itemize}
    \item $a(Q) \neq a(Q')$ if $a(Q)$ and $a(Q')$ are of equal length,
    \item if $a(Q)$ is shorter than $a(Q')$ then $a(Q')$ does not begin with $a(Q)$,
    \item if $a(Q')$ is shorter than $a(Q)$ then $a(Q)$ does not begin with $a(Q')$.
  \end{itemize}
  Such an encoding scheme can be done inductively.  We omit the details although the reader can consult \cite{david} and \cite{jones} for full details.

  If $x \in (Q_0 \cap K) \backslash R_{L'}$, then there there are fewer than $L'$ cubes $Q$ containing $x$ for which $\calF(Q) \neq \emptyset$.  Thus, we see that there must be some $Q \in \Delta$ containing $x$ so that for all $Q' \subset Q$ containing $x$, $a(Q) = a(Q')$, {\it i.e.} the coding stablizes.  We let $a(x) = a(Q)$ for this $Q$.  One can see that $a(Q)$ is a word of at most $L'$ letters.  We then have that $(Q_0 \cap DC(\delta,1/10,16 \ell(Q))) \backslash (R_{L'} \cup \Sigma(c\delta^n) \cup \Lambda(\eta))$ can be decomposed into at most $(T+1)^{L'}$ measurable sets based on what word each point is assigned.
  
  Let $x,y$ be two points of $(Q_0 \cap DC(\delta,1/10,16 \ell(Q))) \backslash (R_{L'} \cup \Sigma(c\delta^n) \cup \Lambda(\eta))$ such that $a(x) = a(y)$.  We want to show that $|\phi(x) - \phi(y)| > k d(x,y)$ for some $k$ independent of $x$ and $y$.  This would show that $\phi$ is biLipschitz on the piece corresponding to $a(x)$.  We may as well suppose $x \neq y$.  Let $S$ be the smallest cube such that $x \in S$ and $y \in 2S$.  If $S \in M_A(\zeta)$, then we get our needed bound from Lemma \ref{l:weak-bilipschitz} once one recalls the definition of $b$ as defined right before the same Lemma.  Otherwise $S \in \Delta(Q_0) \backslash M_A(\zeta)$.  Letting $k$ be so that $S \in \Delta_{k+\ell}$ and $Q,Q' \in \Delta_k$ so that $x \in Q$ and $y \in Q'$, we get that $d(x,y) \geq b \diam S > \diam Q$ and so $Q \neq Q'$.  Thus, $Q' \in \calF(Q)$.  But this is a contradiction of the fact that $a(x) = a(y)$ as the conditions of our encoding scheme ensures that such a situation cannot happen.  This finishes the sketch of the encoding argument.
\end{proof}

\section{Proof of main theorem}

We will need a preliminary lemma.

\begin{lemma} \label{l:cube-decomposition}
  Let $A$ be any bounded and finite measurable set in $K$ and $\epsilon,\lambda > 0$.  There exists some $J \leq k_K$ so that for every $j \leq J$ there exists a finite collection of disjoint cubes $\{Q_k\}_{k=1}^M \subset \Delta_j$ so that $\mu(A \backslash \bigcup_i Q_i) < \epsilon$ and $\mu(Q_i \backslash A) < \lambda \mu(Q_i)$.
\end{lemma}

\begin{proof}
  By Lebesgue's differentiation theorem, there exists some $R > 0$ and some subset $A' \subseteq A$ so that $\mu(A \backslash A') < \epsilon$ and
  \begin{align*}
    \mu(B(x,r) \cap A) > (1-c\lambda) \mu(B(x,r)), \qquad \forall r \in (0,R), x \in A'.
  \end{align*}
  Here, $c > 0$ is some constant depending only on $K$ so that if $Q$ is a cube of diameter no more than $R$ which intersects $A'$ nontrivially, then
  \begin{align*}
    \mu(Q \cap A) > (1-\lambda) \mu(Q).
  \end{align*}
  That such a constant exists easily follows from \eqref{e:cube-diam}, \eqref{e:cube-growth}, and \eqref{e:ball-growth}.

  Thus, we can choose $J \leq k_K$ so that $32 \cdot 16^J \leq R$ and for every $j \leq J$, we let $\{Q_i\}_{i=1}^M \subset \Delta_j$ be the cubes that intersect with $A'$ nontrivially.
\end{proof}

We can now prove the main theorem.

\begin{proof}[Proof of Theorem \ref{t:main}]
  As mentioned in the introduction, we simply need to prove that \eqref{i:main-alberti} implies \eqref{i:main-david} and \eqref{i:main-david} implies (R) in the main theorem.

  Recall that, after Section 2, we are reduced to proving that $K$ is $n$-rectifiable.  As discussed in Section 2, after possibly taking a further countable Borel decomposition, Proposition \ref{p:dc-decomposition} proves the implication of \eqref{i:main-david} from \eqref{i:main-alberti}.  It thus suffices to show that any $DC(\beta,1/10,R)$ of positive measure is $n$-rectifiable.  We now fix such a $\beta \in (0,1)$ and $R > 0$.  We may suppose without loss of generality that $DC(\beta,1/10,R)$ also is bounded and has finite measure.

  Let $\epsilon > 0$.  By an application of Lemma \ref{l:cube-decomposition}, there exists some $j \leq k_K$ and a finite collection of cubes $\{Q_i\}_{i=1}^M \subset \Delta_j$ so that $16^{j+1} \leq R$,
  \begin{align}
    \mu\left(DC(\beta,1/10,R) \backslash \bigcup Q_i\right) &< \frac{\epsilon}{4}, \label{e:approx-DC} \\
    \mu(Q_i \backslash DC(\beta,1/10,R)) &< \frac{\epsilon}{8C \mu(DC(\beta,1/10,R))} \mu(Q_i), \qquad \forall i, \label{e:dense-cube}
  \end{align}
  where $C > 0$ is the constant from Proposition \ref{p:local-decompose}.  Assuming that we had initially chosen $\epsilon < 4C\mu(DC(\beta,1/10,R))$, which we can and will, we then get from \eqref{e:dense-cube} that
  \begin{align}
    \sum_{i=1}^M \mu(Q_i) < 2\mu(DC(\beta,1/10,R)). \label{e:cube-sum-small}
  \end{align}
  By choosing $\eta = \epsilon/(8\mu(DC(\beta,1/10,R)))$ and applying Proposition \ref{p:local-decompose}, for every $Q_i$, there exists a finite family of compact subsets $\{F_{i,j}\}_{j=1}^{m_i}$ of $K$ so that $\phi|_{F_{i,j}}$ is biLipschitz and
  \begin{align*}
    \mu\left((Q_i \cap K) \backslash \bigcup_j F^i_j \right) &\leq \frac{\epsilon}{4M} + \frac{\epsilon}{8 \mu(DC(\beta,1/10,R))} \mu(Q_i) + C\mu(Q_i \backslash DC(\beta,1/10, 16^{j+1})) \\
    &\leq \frac{\epsilon}{4M} + \frac{\epsilon}{8 \mu(DC(\beta,1/10,R))} \mu(Q_i) + C\mu(Q_i \backslash DC(\beta,1/10,R)).
  \end{align*}
  We can apply Proposition \ref{p:local-decompose} because we can use Lemma \ref{l:Sigma-bound}, Lemma \ref{l:Lambda-bound}, and \eqref{e:dense-cube} to prove that $(Q_i \cap K) \backslash (\Sigma(c\beta^n) \cup \Lambda(\eta))$ has positive measure whenever $\epsilon$ is chosen small enough.

  We let $F_{i,j}' = F_{i,j} \cap DC(\beta,1/10,R)$, which clearly also satisfies
  \begin{multline}
    \mu\left((Q_i \cap DC(\beta,1/10,R)) \backslash \bigcup_j F_{i,j}' \right) \\
    \leq \frac{\epsilon}{4M} + \frac{\epsilon}{8 \mu(DC(\beta,1/10,R))}\mu(Q_i) + C\mu(Q_i \backslash DC(\beta,1/10,R)). \label{e:F'-dense}
  \end{multline}

  Thus, $\{F_{i,j}'\}_{i=1,j=1}^{M,m_i}$ are a finite collection of bounded subsets of $DC(\beta,1/10,R)$ on each of which $\phi$ is biLipschitz.  The last step is to bound the size of $DC(\beta,1/10,R) \backslash \bigcup_{i,j} F_{i,j}'$.  We have
  \begin{align*}
    &\mu\left( DC(\beta,1/10,R) \backslash \bigcup_{i,j} F_{i,j}' \right) \\
    &\qquad \leq \mu\left( DC(\beta,1/10,R) \backslash \bigcup Q_i \right) + \sum_{i=1}^M \mu\left( (Q_i \cap DC(\beta,1/10,R)) \backslash \bigcup_j F_{i,j}' \right) \\
    &\qquad \overset{\eqref{e:approx-DC} \wedge \eqref{e:F'-dense}}{<} \frac{\epsilon}{2} + \frac{\epsilon}{8\mu(DC(\beta,1/10,R))} \sum_{i=1}^M \mu(Q_i) + C \sum_{i=1}^M \mu(Q_i \backslash DC(\beta,1/10,R)) \\
    &\qquad \overset{\eqref{e:dense-cube} \wedge \eqref{e:cube-sum-small}}{<} \frac{3\epsilon}{4} + \frac{\epsilon}{8\mu(DC(\beta,1/10,R))} \sum_{i=1}^M \mu(Q_i) \overset{\eqref{e:cube-sum-small}}{<} \epsilon.
  \end{align*}

  As $\epsilon > 0$ was arbitrary, this finishes the proof that $DC(\beta,1/10,R)$ is $n$-rectifiable, which, as mentioned before, proves that $K$ is $n$-rectifiable.  Thus, we have shown that \eqref{i:main-david} implies (R), which finishes the proof of Theorem \ref{t:main}.
\end{proof}

\begin{remark}
        Note that in the proof of (R) using (iii) (and in the proof of Proposition \ref{p:M-carleson} to be proven in Appendix B below), the only property we used of the $\R^n$ target to get a biLipschitz decomposition of $\phi$ is that it is Ahlfors $n$-regular and so its balls satisfy the estimate $|B(x,r)| \geq C r^n$ for some $C > 0$.  This is most crucially used in the proof of Lemma \ref{l:Sigma-bound} to prove that $\mu(Q_0 \backslash DC(\delta,1/10,\ell(Q))$ bounds the measure of $\Sigma(c\delta^n)$.  Specifically, it is used when deducing that $Q \not \in \SI(c\delta^n)$ from \eqref{e:qnotinsi}.  This allowed to us show in the proof of Theorem \ref{t:main} that the volume of the $\SI$ cubes can be made negligible.  Ahlfors regularity was used again in Lemma \ref{l:weak-bilipschitz} to prove \eqref{e:ball-overlap}.  However, this use of Ahlfors regularity can be completely avoided with a little more work.  See the proof of Claim 9.49 and Remark 9.56 in \cite{semmes} for more detail.

%  Thus, for $\beta,\gamma,\epsilon,R > 0$, we can define the sets
%  \begin{multline*}
%    DC(\beta,\gamma,\epsilon,R) = \{x \in K : |B(\phi(x),\beta r) \cap \phi(B(x,r) \cap K)| \geq (1-\epsilon) |B(\phi(x),\beta r)| \\
%    \geq \gamma r^n, \forall r < R\},
%  \end{multline*}

  Thus, we have also proven the following theorem.
  \begin{theorem} \label{t:general-decomp}
    Let $(X,d,\mu)$ a metric measure space such that
    \begin{align*}
      0 < \Theta_*^n(\mu;x) \leq \Theta^{*,n}(\mu;x) < \infty
    \end{align*}
    for a.e. $x$ and $\phi : X \to (Y,\rho)$ Lipschitz such that $(Y,\rho,\Hd^n)$ is Ahlfors $n$-regular.  If $X$ satisfies David's condition a.e. with respect to $\varphi$, then there exists a countable number of Borel sets $U_i \subset X$ such that $\mu\left(X \backslash \bigcup_i U_i\right) = 0$ and $\phi|_{U_i} \colon U_i \to Y$ is biLipschitz.
  \end{theorem}

  Note that we really only needed one direction of the Ahlfors regularity assumption of $\Hd^n$ for $\phi(X)$ (remembering that we do not need Ahlfors regularity in Lemma \ref{l:weak-bilipschitz}).  Thus, we can relax our conditions for $Y$ in Theorem \ref{t:general-decomp} to general metric space targets with the addition of this assumption on $\phi(X)$.  We may further relax this condition so that, for almost every $x \in X$, $\phi(x)$ has a positive lower $n$-density in $\phi(X)$.  This will only require one further countable Borel decomposition of $X$ into the sets where a lower $n$-density estimate holds below a certain radius for a certain lower bound, as has been done above for other similar properties.
\end{remark}

\appendix

\section{Proof of Proposition \ref{p:cubes-1}}
In this appendix, we go over the outline of the proof of Proposition \ref{p:cubes-1}.  The proof of this theorem is essentially that of Theorem 11 of \cite{christ}, which a good understanding of will be helpful.  In the areas where we need to make nontrivial changes, we will go over the changes carefully.  We may suppose without loss of generality that $\mu(K) > 0$.  Our lemmas and propositions will be numbered A.X' where X is the numbering of the corresponding lemma or proposition in \cite{christ}.  Let $k_K \in \Z$ be maximal so that $16^{k_K + 2} \leq R$.

We begin by defining a collection of nets $\{z_\alpha^k : k \leq k_K, \alpha \in I_k\} \subseteq K$ so that
\begin{align*}
  d(z_\alpha^k, z_\beta^k) \geq 16^k, \qquad \forall \alpha \neq \beta
\end{align*}
and for each $k$ and $x \in K$, there exists $\alpha \in I_k$ so that $d(x,z_\alpha^k) < 16^k$.

\begin{customlemma}{A.13'}
  There exists at least one partial ordering of the set $\{(k,\alpha) : k \leq k_K, \alpha \in I_k\}$ so that
  \begin{enumerate}[(a)]
    \item $(k,\alpha) \leq (\ell,\beta) \Rightarrow k \leq \ell$,
    \item For each $(k,\alpha)$ and $\ell \in \{k,...,k_K\}$, there exists a unique $\beta$ such that $(k,\alpha) \leq (\ell,\beta)$,
    \item $(k-1,\beta) \leq (k,\alpha) \Rightarrow d(z_\alpha^k, z_\beta^{k-1}) < 16^{k-1}$,
    \item $d(z_\alpha^k,z_\beta^{k-1}) < \frac{1}{2}16^{k-1} \Rightarrow (k-1,\beta) \leq (k,\alpha)$.
  \end{enumerate}
\end{customlemma}
This partial ordering is easily constructed by using the net properties of the $z_\alpha^k$.  There may be some choices that need to be made, which would reflect the fact that there could be multiple partial orderings, but these choices will not affect the result.  We let the reader consult the proof of Lemma 13 of \cite{christ} for the complete details.

Choose one such partial ordering.  We then define
\begin{align*}
  Q_\alpha^k = \bigcup_{(\ell,\beta) \leq (k,\alpha)} B(z_\beta^\ell,16^\ell/8).
\end{align*}

Properties 2-4 of Proposition \ref{p:cubes-1} are easily verifiable by our construction of the $Q_\alpha^k$ and so we omit the proofs.  We actually get the improved estimate
\begin{align}
  Q_\alpha^k \subseteq B\left(z_\alpha^k, \frac{3}{2} \cdot 16^k\right). \label{e:improved-diam}
\end{align}
For full details, consult the proof of Theorem 11 of Lemma 13.  Property 6 follows easily from Property 4, \eqref{e:ball-growth-hyp}, and the fact that $16^{k_K + 2} \leq R$.

The proof of Property 1 is also easy.  Let $E = \bigcup_\omega Q_\omega^k$ for some $k \leq k_K$ and $x \in K \backslash E$.  Let $n \leq k$.  Then there exists some $z_\alpha^n$ such that $d(z_\alpha^n,x) < 16^n$.  Note that $B(z_\alpha^n,16^n/8) \subset E$ by construction.  Thus, $B(z_\alpha^n,16^n/8) \subseteq B(x,16^{n+1})$ and as $x,z_\alpha^n \in K$, we get that
\begin{align*}
  \frac{\mu(E \cap B(x,16^{n+1}))}{\mu(B(x,16^{n+1}))} \geq \frac{\mu(B(z_\alpha^n,16^n/8))}{\mu(B(x,16^{n+1}))} \geq c > 0,
\end{align*}
for some $c > 0$ depending only on the constants of \eqref{e:ball-growth-hyp}.  As this holds for all $n \leq k$ and $x \notin E$, we get by Lebesgue density theorem that $\mu(K \backslash E) = 0$.

It remains to prove Property 5 (known as the small boundaries condition), which requires more substantial changes to the proof of the analogous property of Theorem 11 of \cite{christ}.  We will need the following lemma:
\begin{customlemma}{A.17'} \label{l:modified-17}
  For any $\epsilon > 0$, there exists $N \in \N$ such that for every $Q_\alpha^k$,
  \begin{align*}
    \mu\{x \in Q_\alpha^k : \exists \sigma \in I_{k-N} \text{ so that } x \in Q_\sigma^{k-N}, \dist(Q_\sigma^{k-N},X \backslash Q_\alpha^k) < 50 \cdot 16^{k-N}\} < \epsilon \mu(Q_\alpha^k).
  \end{align*}
\end{customlemma}

\begin{proof}
  The proof of the modified Lemma 17 is similar to the proof of the original.  Let $x \in Q_\sigma^{k-N} \cap Q_\alpha^k$ as above for some $N$ to be determined.  Then we have that there exists a unique chain of cubes
  \begin{align*}
    Q^{k-N}_\sigma = Q_{\sigma_{k-N}}^{k-N} \subset Q_{\sigma_{k-N+1}}^{k-N+1} \subset \cdots \subset Q_{\sigma_k}^k = Q_\alpha^k.
  \end{align*}
  We let $z_j = z_{\sigma_j}$ be the points as in Property 4.
  
  We claim that there exists some absolute constant $\epsilon_1 > 0$ so that $d(z_i,z_j) \geq 2 \epsilon_1 16^j$ when $k-N+3 \leq i < j \leq k$.  Suppose not.  Then
  \begin{multline*}
    \dist(z_j, X \backslash Q_\alpha^k) \leq \dist(x,X \backslash Q) + d(z_j,x) \leq 100 \cdot 16^{k-N} + d(z_j,z_i) + d(z_i,x) \\
    \overset{\eqref{e:improved-diam}}{\leq} 100 \cdot 16^{k-N} + 2 \epsilon_1 16^j + \frac{3}{2} \cdot 16^i = (*).
  \end{multline*}
  By choosing $\epsilon_1$ smaller than some absolute constant and using the fact that $i \geq k-N+3$, we get that
  \begin{align*}
    (*) < \frac{1}{8} \cdot 16^j.
  \end{align*}
  But this is a contradiction of the fact that $B(z_j,16^j/8) \subseteq Q_\alpha^k$.

  Now for each $x \in \bigcup\{Q_\sigma^{k-N} \subset Q_\alpha^k : d(Q_\sigma^{k-N},X \backslash Q_\alpha^k) < 50 \cdot 16^{k-N}\}$, we can construct a chain of cubes and $z_{\beta(x,j)}$ as above for $k-N \leq j \leq k$.  Let $S_j$ be the collection of all points $z_{\beta(x,j)}$ for all such $x$.  Let $G_j = \bigcup_{z \in S_j} B(z,\epsilon_1 16^j)$.  Then we have that the $G_j$ are disjoint.

  Let
  \begin{align*}
    E = \{x \in Q_\alpha^k : \exists \sigma \in I_{k-N} \text{ so that } x \in Q_\sigma^{k-N}, \dist(Q_\sigma^{k-N},X \backslash Q_\alpha^k) < 50 \cdot 16^{k-N}\}.
  \end{align*}
  We have for any $k \leq j \leq k-N+3$ that
  \begin{align*}
    \mu(E) &\leq \mu \left( \bigcup_{z \in S_{k+N}} B(z, 16^{k+N+1}) \right) \\
    &\leq C' \sum_{z \in S_{k+N}} \mu(B(z,\epsilon_1 16^{k+N})) \\
    &= \sum_{w \in S_j} \sum_{z \in S_{k+N}, z \leq w} \mu(B(z,\epsilon_1 16^{k+N})) \\
    &\leq C' \sum_{w \in S_j} \mu(B(w,16^{j+1})) \\
    &\leq C' \sum_{w \in S_j} \mu(B(w,\epsilon_1 16^j)) \leq C' \mu(G_j).
  \end{align*}
  Here, the first inequality comes from Property 4 of the cubes.  The second and fourth inequalities come from the \eqref{e:ball-growth-hyp}, the fact that the $S_j$ sets are in $K$, and the disjointness of the $B(z,\epsilon_1 16^{k+N})$.  The third inequality comes from Property 4 of the cubes.

  As the $G_j$ are all disjoint, one gets that
  \begin{align*}
    \mu(Q) \geq \sum_{j=k}^{k-N+3} \mu(G_j) \geq (N-3) C'^{-1} \mu(E).
  \end{align*}
  Thus, choosing $N > C'\epsilon^{-1}+3$ allows us to conclude that $\mu(E) < \epsilon \mu(Q)$.
\end{proof}

The proof of Property 5 now follows the proof in Theorem 11 of \cite{christ}.  One defines
\begin{align*}
  E_j(Q_\alpha^k) = \{Q_\beta^{k-j} \subset Q_\alpha^k : \dist(Q_\beta^{k-j},X \backslash Q_\alpha^k) \leq 50 \cdot 16^{k-j}\}
\end{align*}
and $e_j(Q_\alpha^k) = \bigcup_{Q \in E_j(Q_\alpha^k)} Q$.  One can verify that if $x \in Q_\alpha^k \cap K$ so that $d(x,X \backslash Q_\alpha^k) < \tau 16^k$, then $x \in e_j(Q_\alpha^k)$ where $16^{k-j-1} \geq \tau 16^k$.  For full details, the reader can consult the second paragraph of proof of Lemma 17 of \cite{christ}.

Thus, to prove Property 5, it suffices to prove that that there is some $C' > 0$ so that
$$\mu(e_j(Q_\alpha^k) \cap K) \leq C' 16^{-j\eta} \mu(Q_\alpha^k), \qquad \forall \alpha, k, \forall j \geq 0.$$
Lemma \ref{l:modified-17} says that there exists some $J \geq 0$ so that
\begin{align}
  \mu(e_J(Q_\alpha^k)) \leq \frac{1}{2} \mu(Q_\alpha^k), \qquad \forall k,\alpha. \label{e:boundary-drop}
\end{align}
Define $F_n(Q_\alpha^k)$ to be a collection of all $Q_\beta^{k-nJ} \subset Q_\alpha^k$ as follows.  For $n = 1$, $F_1(Q_\alpha^k) = E_J(Q_\alpha^k)$.  The we define inductively
\begin{align*}
  F_{\ell+1}(Q_\alpha^k) = \bigcup_{Q_\beta^{k-\ell J} \in F_\ell(Q_\alpha^k)} E_J(Q_\beta^{k-\ell J}).
\end{align*}
Letting $f_n(Q_\alpha^k) = \bigcup_{S \in F_n(Q_\alpha^k)} S$, we get by iteration of \eqref{e:boundary-drop} that
\begin{align*}
  \mu(f_m(Q_\alpha^k)) \leq 2^{-m} \mu(Q_\alpha^k).
\end{align*}
Thus, we get
\begin{align*}
  \mu(e_{mJ}(Q_\alpha^k)) \leq 2^{-m} \mu(Q_\alpha^k) = 16^{-\eta mJ} \mu(Q_\alpha^k), \qquad \forall m \geq 0.
\end{align*}
This finishes the proof. \qed

\section{Proof of Proposition \ref{p:M-carleson}}
\subsection{Introduction}
The proof of Proposition \ref{p:M-carleson} closely follows the proof of Proposition 7.8 in \cite{semmes}, which, needless to say, a good understanding of will be necessary.  Most of the proof will only require superficial changes.  Thus, for convenience, we will not go into much details in these parts although we will try to prove a general overview of how the proof works.  There are some parts where we will have to make nontrivial modifications; most notably, we have to add an extra condition to a stopping time process defined in \cite{semmes}.  We will go through these modifications carefully.  Our lemmas and propositions will be numbered B.X.Y' where X.Y is the numbering of the corresponding lemma or proposition in \cite{semmes}.

Recall that we have $K \subset X$ satisfying \eqref{e:ball-growth}, a collection of cubes $\Delta$ satisfying Proposition \ref{p:cubes-1}, and a 1-Lipschitz map $\phi : K \to \R^n$.  We will recall further necessary concepts from \cite{semmes} in the relevant subsections below.  Also recall the families of cubes $\LD(\eta)$ from Section 4, which was not in \cite{semmes}.  This family of cubes (or more precisely, the complement of this family) will play a big role in establishing our $K$-Carleson bound.

Recall a nonempty collection cubes $S \subseteq \Delta$ is called a {\it stopping time region} if there is a top cube $Q(S) \in S$ so that for every other $Q \in S$ such that $Q \subset Q(S)$ and $Q' \in \Delta$ such that $Q \subseteq Q' \subseteq Q(S)$, then $Q' \in S$.  The set of bottom cubes of a stopping time region $S$ is
\begin{align*}
  b(S) = \{Q \in \Delta : Q \subseteq Q(S), Q \notin S, \text{and $Q$ is maximal with respect to these properties} \}.
\end{align*}
Note that these are disjoint cubes contained in $Q(S)$ if they exist ($b(S)$ may in fact be an empty set).

We quickly go over a general outline of this part of the appendix.  In some of the subsections below, we will give an more in depth overview.  Section B.2 establishes some important technical lemmas that allows us to determine when a collection of cubes is $K$-Carleson.  Most of the proofs of the original lemmas can be superficially modified to get our needed result.

Sections B.3 gives our first preliminary decomposition of the set of subcubes of some cube $Q_0$ into stopping time regions for which our map $\phi$ has good measure preserving properties on each stopping time region.  The stopping time regions are shown to not be ``too much'' and, together with $\LD(\eta)$, cover a large portion of $Q_0$.  This is only a partial decomposition of the subcubes of a specified cube.  See Proposition \ref{p:3.6} for the precise statement and Section B.3 for an overview of the proof.

In Section B.4, we iterate the partial stopping time decomposition of Section B.3 to get a complete stopping time decomposition of the entire set of subcubes of $Q_0$.  As before, $\phi$ has good measure preserving properties on each stopping time region.  It will be shown that the number of stopping time regions will not be too large in the sense that all the top cubes satisfy a $K$-Carleson condition.  There will be a junk set of cubes that will be the cubes of $\SI$ and $\LD$ that satisfy a Carleson condition.  See Proposition \ref{p:4.2} for the precise statement and Section B.4 for an overview of the proof.

Section B.5 refines the stopping time region decomposition of Section B.4 so that each stopping time region becomes good (the definition of a good stopping time region is in the beginning of Section B.5).  This requires a decompositition of each stopping time region from Section B.4 into a collection of good stopping time regions.  This will increase the number of stopping time regions so it will be necessary to show that the top cubes of the collection of good stopping time regions still are $K$-Carleson.  Note that $\phi$ still has good measure preserving properties on the good stopping time regions because they are subsets of the stopping time regions of Section B.4.  The junk set of $\SI$ and $\LD$ cubes remain untouched.  See Proposition \ref{p:5.5} for the precise statement and Section B.5 for an overview of the proof.

In Section B.6, we improve on the measure preserving properties of $\phi$.  See \eqref{e:G-sigma-defn} for the definition of $\calG(\sigma)$, the specific improvement on measure preservation that we use.  We show that most subcubes of $Q_0$ belong to $\calG(\sigma)$.  Specifically, if we consider $G_2$ to be all the cubes in all the stopping time regions of the decomposition of Section B.5, then we show that $G_2 \backslash \calG(\sigma)$ is a $K$-Carleson collection of subcubes.  See Proposition \ref{p:6.13} and Section B.6 for an overview of the proof.

Finally, we give the proof of Proposition \ref{p:M-carleson} in Section B.7 by showing that most cubes of $\calG(\sigma)$ are in $M_A(\zeta)$ in the sense that $\calG(\sigma) \backslash M_A(\zeta)$ is $K$-Carleson.  This together with the result of Section B.6 and the fact that the junk set outside of $G_2$ is in $\SI$ and $\LD$ gives us our needed result.  See Section B.7 for an overview of the proof.

\subsection{Section 2 of \cite{semmes}: Lemmas}
In this subsection, we recall some preliminary terminology and lemmas.  Most of the lemmas of Section 2 of \cite{semmes} establish Carleson bounds.  We will convert them to establishing $K$-Carleson bounds.

Fix a cube $Q_0$ and let $\calE$ a family of cubes.  For $x \in Q_0$, we let $N(x) = N_{Q_0}(x)$ denote the number of cubes in $\calE$ that contain $x$.  Note that all such cubes are then subsets of $Q_0$.

The following lemma will be very important in establishing Carleson bounds.
\begin{customlemma}{B.2.28'} \label{l:2.28}
  Let $\calE$ be a family of cubes in $\Delta$ and suppose that there are positive constants $k,\lambda$ such that
  \begin{align}
    \mu(\{x \in Q \cap K : N_Q(x) > L\}) \leq (1-\lambda) \mu(Q \cap K), \qquad \forall Q \in \Delta. \label{e:N-small}
  \end{align}
  Then $\calE$ is a $K$-Carleson set with constants depending only on $L$ and $\lambda$.
\end{customlemma}
The estimate \eqref{e:N-small} allows one to show that the distribution function
\begin{align*}
  \lambda(t) = \mu(\{x \in Q_0 \cap K : N_{Q_0}(x) > t\})
\end{align*}
decays exponentially, giving us a bound on the Carleson summation.  The full proof of the original Lemma 2.28 in \cite{semmes} requires only superficial modifications to give us Lemma \ref{l:2.28} and so is omitted.  Note that following that proof actually gives the stronger statement that
\begin{align*}
  \sum_{Q' \in \calE, Q' \subseteq Q} \mu(K \cap Q') \leq C \mu(K \cap Q), \qquad \forall Q \in \Delta,
\end{align*}
but we will not need this.

Let $\eta > 0$.  If a family of cubes $\calE$ is a $K$-Carleson set with constant $C$ and does not contain any cubes for which $\mu(Q \cap K) < \eta \mu(Q)$, then it follows easily from the definition of ($K$-)Carleson sets that $\calE$ is actually a Carleson set with constant $C/\eta$.

Given a family of cubes $\calE$, we let $\calE_A$ denote the set of cubes $Q$ such that $Q$ is an $A$-neighbor of some $Q' \in \calE$. 

We will need the original version of Lemma 2.32 of \cite{semmes}.
\begin{customlemma}{B.2.32} \label{l:2.32}
  Let $\calE$ be a Carleson set with constant $C$.  Then $\calE_A$ is also an Carleson set with some constant $C'$ increased by a factor depending on $A$ and $K$.
\end{customlemma}

We will also need the following $K$-Carleson version.
\begin{customlemma}{B.2.32'} \label{l:2.32'}
  Let $\eta > 0$ and $\calE$ be a $K$-Carleson set with constant $C$ so that $\mu(Q \cap K) \geq \eta \mu(Q)$ for all $Q \in \calE$.  Then $\calE_A$ is also an $K$-Carleson set with some constant $C'$ increased by a factor depending on $K$, $A$, and $\eta$.
\end{customlemma}

\begin{proof}
  As $\calE$ does not contain cubes for which $\mu(Q \cap K) < \eta \mu(Q)$, we have that it is a Carleson set of constant depending on $\eta$ and the $K$-Carleson constant of $\calE$.  We then apply Lemma \ref{l:2.32} to get that $\calE_A$ is Carleson with constant increased by a factor depending on $A$, $K$, and $\eta$.  As Carleson sets are $K$-Carleson, we are done.
\end{proof}

Given a $T \in \Delta$, we let
\begin{align*}
  C_A(T) = \{Q \in \Delta : ~&\text{there is a cube} ~Q' \in \Delta ~\text{such that} ~Q ~\text{and} ~Q' ~\text{are neighbors,}\\
  &\text{and either} ~Q \subseteq T, Q' \not\subseteq T, ~\text{or} ~Q' \subseteq T, Q \not\subseteq T \}.
\end{align*}
See (2.33) of \cite{semmes}.

\begin{customlemma}{B.2.34'} \label{l:2.34}
  There is a constant $D$ so that
  \begin{align*}
    \sum_{Q \in C_A(T)} \mu(Q \cap K) \leq D \mu(T),
  \end{align*}
  where $D$ depends only on $A$ and $K$.
\end{customlemma}
The proof requires only superficial modifications of the proof of the original Lemma 2.34 in \cite{semmes} and so is omitted.  The small boundaries property of the cubes (Property 5 of Proposition \ref{p:cubes-1}) is used here.  Note by \eqref{e:small-boundaries} that our cubes have the small boundary property only when restricted to $K$, but this is exactly what is needed for the lemma.

We now move to the Carleson set composition lemmas of \cite{semmes}.  Note that $K$-Carleson conditions do not compose because the $K$-Carleson upper bound does not take into account intersections with $K$.  Thus, we will need that there exists some $\eta > 0$ so that all our cubes in the following lemmas satisfy $\mu(Q \cap K) \geq \eta \mu(Q)$.  We call these cubes high density cubes.  They help us turn $K$-Carleson bounds into Carleson bounds, which will lead to nice composition properties.
\begin{customlemma}{B.2.44'} \label{l:2.44}
  Let $\eta > 0$ and $\calX$ be a collection of cubes that is a $K$-Carleson set and for which each $Q \in \calX$ satisfies $\mu(Q \cap K) \geq \eta \mu(Q)$.  Define $\hat{\calX}_A$ by
  \begin{align*}
    \hat{\calX}_A = \bigcup_{T \in \calX} C_A(T).
  \end{align*}
  Then $\hat{\calX}_A$ is also a $K$-Carleson set, with a constant depending only on $K$, $\eta$, and the $K$-Carleson constant for $\calX$.
\end{customlemma}

\begin{proof}
  Fix some $Q_0 \in \Delta$.  As in the proof of Lemma 2.44 of \cite{semmes}, we get that
  \begin{align*}
    \sum_{R \in \hat{\calX}_A, R \subseteq Q_0} \mu(R \cap K) \leq \sum_{R \in C_A(Q_0)} \mu(R \cap K) + \sum_{T \in \calX(Q_0)} \sum_{R \in C_A(T)} \mu(R \cap K).
  \end{align*}
  See the argument before equation (2.47) of \cite{semmes}.  Using Lemma \ref{l:2.34}, we can convert this to
  \begin{align*}
    \sum_{R \in \hat{\calX}_A, R \subseteq Q_0} \mu(R \cap K) \leq C \mu(Q_0) + C \sum_{T \in \calX(Q_0)} \mu(T).
  \end{align*}
  As $\calX$ is $K$-Carleson and does not contain cubes for which $\mu(Q \cap K) < \eta \mu(Q)$, we get that
  \begin{align*}
    \sum_{R \in \hat{\calX}_A, R \subseteq Q_0} \mu(R \cap K) \leq C \mu(Q_0) + \frac{C}{\eta} \sum_{T \in \calX(Q_0)} \mu(T \cap K) \leq C' \mu(Q_0),
  \end{align*}
  where $C'$ depends on the previous $C$ and $\eta$.
\end{proof}

\begin{customlemma}{B.2.50'} \label{l:2.50}
  Let $\eta > 0$ and $\calF$ be a family of stopping-time regions that are disjoint as subsets of $\Delta$.  Assume that the collection of top cubes $\{Q(S) : S \in \calF\}$ is a $K$-Carleson set with constant $C_1$ and does not contain any cubes for which $\mu(Q \cap K) < \eta \mu(Q)$.  Suppose for each $S \in \calF$ we have a collection of cubes $\calE(S) \subseteq S$ that is a $K$-Carleson set with constant $C_2$.  Then the union
  \begin{align*}
    \calE^* = \bigcup_{S \in \calF} \calE(S)
  \end{align*}
  is a $K$-Carleson set with constant depending only on $C_1$, $C_2$, and $\eta$. 
\end{customlemma}
The proof requires only superficial modifications of the proof of the original Lemma 2.50 in \cite{semmes} once one takes into account that $\{Q(S) : S \in \calF\}$ contains only high density cubes and so is actually also a Carleson set.  Thus, the proof is omitted.

\begin{customlemma}{B.2.58'} \label{l:2.58}
  Let $\eta > 0$ and $\calF$ be a family of stopping time regions that are disjoint as subsets of $\Delta$.  For each $S \in \calF$ set
  \begin{align*}
    S_A = \{ Q \in S : Q' \in S ~\text{whenever} ~Q \text{ and } Q' \text{ are neighbors }\},
  \end{align*}
  and set
  \begin{align*}
    \calB_A = \bigcup_{S \in \calF} (S \backslash S_A).
  \end{align*}
  If the collection of top cubes $\{Q(S) : S \in \calF\}$ is a $K$-Carleson set and satisfies $\mu(Q \cap K) \geq \eta \mu(Q)$, then $\calB_A$ is also a $K$-Carleson set with constant depending only on the $K$-Carleson constant for $\{Q(S)\}_{S \in \calF}$, $\eta$, $A$, and $K$.
\end{customlemma}
The proof requires only superficial modifications of the proof of the original Lemma 2.58 in \cite{semmes} once one takes into account that $\{Q(S) : S \in \calF\}$ contains no cubes of low density and so is actually also a Carleson set.  Note that the bottom cubes of a stopping time region $b(S)$ are still Carleson with constant 1.  Thus, the proof is omitted.

\subsection{Section 3 of \cite{semmes}}
The following gives our initial incomplete stopping time region decomposition of $\Delta(Q_0)$ on which $\phi$ has the good measure preserving condition of (c).  There are not too many stopping time regions in the sense of (e), but the area covered by stopping time regions isn't too small in the sense of (a) and (b).  Crucially, we require that all the stopping time regions consist of cubes of high density.

\begin{customproposition}{B.3.6'} \label{p:3.6}
  Let $\tau,\delta,\eta > 0$ and $Q_0 \in \Delta$ such that $|\phi(Q_0 \cap K)| \geq \delta \mu(Q_0)$ and $\mu(Q_0 \cap K) \geq \eta \mu(Q_0)$.  There exist constants $k,\alpha > 0$ depending only on $\tau$, $\delta$, and $K$ so that the following is true.  There exists a family $\calF$ of pairwise-disjoint stopping time regions of $\Delta$ and a measurable subset $E(Q_0) \subseteq Q_0 \cap K$ with the following properties:
  \begin{enumerate}[(a)]
    \item $\mu(E) \geq \alpha \mu(Q_0 \cap K)$,
    \item if $Q \in \Delta$ satisfies $Q \subseteq Q_0$ and $Q \cap E \neq \emptyset$, then either $Q$ lies in $S$ for some stopping-time region $S \in \calF$ or $Q \in \LD(\eta)$,
    \item if $Q \in S$ and $S \in \calF$, then $Q \subseteq Q_0$, and
    \begin{align*}
      (1 + \tau)^{-1} \frac{|\phi(Q(S) \cap K)|}{\mu(Q(S))} \leq \frac{|\phi(Q \cap K)|}{\mu(Q)} \leq (1 + \tau) \frac{|\phi(Q(S) \cap K)|}{\mu(Q(S))},
    \end{align*}
    \item $|\phi(Q(S) \cap K)| \geq \delta \mu(Q(S))$ for all $S \in \calF$,
    \item for each $x \in K$, there are at most $k$ choices of $S \in \calF$ such that $x \in Q(S)$,
    \item $\mu(Q \cap K) \geq \eta \mu(Q)$ for all $Q \in S$ and $S \in \calF$.
  \end{enumerate}
\end{customproposition}

\begin{proof}
  We run the stopping time process of Section 3 of \cite{semmes} on the subcubes of $Q_0$ but with an extra stopping time condition.  Specifically, we stop at a cube $Q \subseteq Q_0$ if any of the following conditions are satisfied:
  \begin{align}
    \frac{|\phi(Q \cap K)|}{\mu(Q)} &< (1+\tau)^{-1} \frac{|\phi(Q_0 \cap K)|}{\mu(Q_0)}, \label{e:stopping-small} \\
    \frac{|\phi(Q \cap K)|}{\mu(Q)} &> (1+\tau) \frac{|\phi(Q_0 \cap K)|}{\mu(Q_0)}, \label{e:stopping-large} \\
    \mu(K \cap Q) &< \eta \mu(Q). \label{e:stopping-LD}
  \end{align}
  Note that the first two conditions may not necessarily be disjoint from the third.  We start with $Q_0$ and only keep the children of $Q_0$ that do not satisfy \eqref{e:stopping-small}, \eqref{e:stopping-large}, or \eqref{e:stopping-LD}.  We then apply the process to the kept children and iterate the process on all the kept children.  This gives us one stopping time region $S_0$, which we put into the singleton family $\calF_0$.  We then look at the bottom cubes $b(S_0)$.  For each $Q \in b(S_0)$ that satisfies \eqref{e:stopping-large} but not \eqref{e:stopping-LD}, we repeat this stopping time process to get another family of stopping time regions $\calF_1$.  Here, we replace the $Q_0$ in the stopping time conditions with the relevant cube of $b(S_0)$.  We repeat again the process for the bottom cubes of all stopping time regions in $\calF_1$ that satisfy \eqref{e:stopping-large} but not \eqref{e:stopping-LD} to get another family of stopping time regions $\calF_2$.  We keep repeating this process over and over to more families $\calF_3,\calF_4,\calF_5...$  Our final family of stopping time regions will be $\calF = \bigcup_{i=0}^\infty \calF_i$.

  Properties (c) and (d) are immediately verifiable.  Property (e) comes from the fact as $\phi$ is 1-Lipschitz, we have that
  \begin{align}
    |\phi(Q \cap K)| \leq \Hd^n(Q \cap K) \overset{\eqref{e:mu-comparison}}{\leq} C2^n \mu(Q \cap K) \leq C2^n \mu(Q), \qquad \forall Q \in \Delta. \label{e:lipschitz-volume}
  \end{align}
  Thus, if a point is contained in $k$ stopping times regions for $k$ large enough, then \eqref{e:stopping-large} must happen too many times and, and, remembering $|\phi(Q_0 \cap K)| \geq \delta \mu(Q_0)$, we will contradict \eqref{e:lipschitz-volume}.  Property (f) is also immediate from the condition of the stopping time.  For more thorough detail, see the analogous proof in \cite{semmes}.

  To construct $E$, for $S \in \calF$, let $b_1(S)$ denote all the cubes of $b(S)$ that satisfy \eqref{e:stopping-small} but not \eqref{e:stopping-LD}.  If we define
  \begin{align*}
    E = (Q_0 \cap K) \backslash \bigcup_{S \in \calF} \bigcup_{Q \in b_1(S)} Q,
  \end{align*}
  then we see that $E$ satisfies Property b.  Indeed, the stopping time process completely terminates only when either \eqref{e:stopping-small} or \eqref{e:stopping-LD} is satisfied.  Thus, if $Q \cap E \neq \emptyset$, then either $Q \subseteq Q'$ for some maximal $Q'$ satisfying \eqref{e:stopping-LD}, $Q \supsetneq Q'$ for some maximal $Q'$ satisfying \eqref{e:stopping-LD}, or $Q$ is disjoint from all such $Q'$.  The first case, $Q \in \LD(\eta)$.  In the second case, we have that $Q' \subsetneq Q \subseteq Q_0$.  From the construction, one then see that $Q \in S$ for some $S \in \calF$ as the fact that $Q'$ was maximal means that $Q'$ was the when the process terminated.  In the third case, we see that $Q$ then contains some $S \in \calF$ as the process never terminated inside $Q$ and so $Q \in \calF$ by similar reasoning of the second case.  In all three cases, we see that $Q$ satisfies Property (b).

  It remains to lower bound $\mu(E)$ as in Property (a).  We need the following lemma.
  \begin{customlemma}{B.3.16'} \label{l:3.16}
    Let $Q$ be a cube in $\Delta$ and $\{Q_i\}_i$ be a disjoint family of subcubes for which
    \begin{align*}
      \frac{|\phi(Q_i \cap K)|}{\mu(Q_i)} \leq (1 + \tau)^{-1} \frac{|\phi(Q \cap K)|}{\mu(Q)}, \qquad \forall i.
    \end{align*}
    Then
    \begin{align*}
      \mu\left( (Q \cap K) \backslash \bigcup_i Q_i \right) \geq \frac{\tau}{1+\tau} |\phi(Q \cap K)|.
    \end{align*}
  \end{customlemma}
  The proof requires only superficial modifications of the proof of the original Lemma 3.16 in \cite{semmes} and so will be omitted.

  Let $S \in \calF$ and set $E_0(S) = (Q(S) \cap K) \backslash \bigcup_{R \in b_1(S)} R$.  We get from Lemma \ref{l:3.16} that
  \begin{align}
    \mu(E_0(S)) \geq \frac{\tau}{1+\tau} |\phi(Q(S) \cap K)| \geq \frac{\tau}{1+\tau} \delta \mu(Q(S) \cap K), \label{e:E0-decrease}
  \end{align}
  where we used Property (d) in the last inequality.  This is the analogue of equation (3.26) of \cite{semmes}.  The rest of the proof of Property (a) only requires superficial modifications of the proof of the original Property (a).  We define
  \begin{align*}
    E_j = (Q_0 \cap K) \backslash \bigcup_{i=0}^j \bigcup_{S \in \calF_i} \bigcup_{R \in b_1(S)} R.
  \end{align*}
  We get then that
  \begin{align*}
    E_{j+1} = E_j \backslash  \bigcup_{S \in \calF_{j+1}} \bigcup_{R \in b_1(S)} R
  \end{align*}
  and $E = E_k$ by Property (e).  The inequality \eqref{e:E0-decrease} allows us to estimate (with a little work)
  \begin{align*}
    \mu(E_{j+1}) \geq \frac{\tau}{1+\tau} \delta \mu(E_j).
  \end{align*}
  This easily gives us our needed lower bound for $\mu(E) = \mu(E_k)$.  See the proof of Lemma 3.6 in \cite{semmes} for more details.
\end{proof}

As was proven in Remark 3.46 of \cite{semmes}, the set $G = \bigcup_{S \in \calF} S$ is itself a stopping time region.

\subsection{Section 4 of \cite{semmes}}
The following proposition gives our first complete stopping time decomposition of $\Delta(Q_0)$.  The completeness is given in (b) where we are allowed a junk set as defined by (a), (f), and (g).  Crucially, there are not too many stopping time regions as expressed in (e).  As before, the stopping time regions consist only of cubes of high density.  We do this by essentially repeating the stopping time decomposition of Proposition \ref{p:3.6} over and over until we have exhausted all cubes.

\begin{customproposition}{B.4.2'} \label{p:4.2}
  Let $Q_0 \in \Delta$ and fix $\delta,\tau,\eta > 0$.  There exists a constant $k_1$ depending only on $K$, $\delta$, and $\tau$, as well as a family $\calF_1$ of stopping-time regions in $\Delta$ and two collections $\{Q_i\}_{i \in I}$ and $\{P_j\}_{j \in J}$ of cubes in $M$ so that the following are true:
  \begin{enumerate}[(a)]
    \item the $Q_i$'s and $P_j$'s together are pairwise disjoint subcubes of $Q_0$ and the stopping time regions $\calF$ are pairwise disjoint subsets of $\Delta(Q_0)$,
    \item if $R \in \Delta(Q_0)$ then either $R \subseteq Q_i$ for some $i \in I$, $R \subseteq P_j$ for some $j \in J$, or $R \in S$ for some $S \in \calF_1$ (but not more than one),
    \item if $Q \in S$ and $S \in \calF_1$, then
    \begin{align*}
      (1 + \tau)^{-1} \frac{|\phi(Q(S) \cap K)|}{\mu(Q(S))} \leq \frac{|\phi(Q \cap K)|}{\mu(Q)} \leq (1 + \tau) \frac{|\phi(Q(S) \cap K)|}{\mu(Q(S))}, 
    \end{align*}
    \item $|\phi(Q(S) \cap K)| \geq \delta \mu(Q(S))$ for all $S \in \calF_1$,
    \item the family of cubes $\{Q(S) : S \in \calF_1\}$ is a $K$-Carleson set with constant $k_1$,
    \item $|\phi(Q_i \cap K)| < \delta \mu(Q_i), \qquad \forall i \in I$,
    \item $\mu(P_j \cap K) < \eta \mu(P_i), \qquad \forall j \in J$.
    \item $\mu(Q \cap K) \geq \eta \mu(Q)$ for all $Q \in S$ and $S \in \calF_1$.
  \end{enumerate}
\end{customproposition}

\begin{proof}
  We may assume that $|\phi(Q_0 \cap K)| \geq \delta \mu(Q_0)$ and $\mu(Q_0 \cap K) \geq \eta \mu(Q_0)$ as otherwise there is nothing to do.  As in \cite{semmes}, the union of all the stopping time regions $S \in \calF$ where $\calF$ is the family of stopping time regions of Proposition \ref{p:3.6} is a stopping time region itself.  Thus, we apply Proposition \ref{p:3.6} to $Q_0$ to get a family of stopping time regions $\calF(Q_0)$ and let $G$ denote the stopping time region that is the union of all cubes of $\calF$.  Let $b(G)$ denote the bottom cubes of $G$.  By construction, if $Q \in b(G)$, then $Q$ has to satisfy at least one of \eqref{e:stopping-small} or \eqref{e:stopping-LD}.  All the cubes that satisfy \eqref{e:stopping-LD} we put into $\{P_i\}$.  For the other cubes, we check to see if
  \begin{align*}
    |\phi(Q \cap K)| < \delta \mu(Q).
  \end{align*}
  If so, we put it into $\{Q_i\}$.  Any remaining bottom cube $Q'$ satisfy $\phi(Q' \cap K) \geq \delta \mu(Q')$ and $\mu(Q' \cap K) \geq \eta \mu(Q')$ and so we apply the stopping time process of Proposition \ref{p:3.6} on each of these to get more families of stopping time regions $\calF(Q')$.  We continue this way forever or until we run out of cubes.  We let $\calF_1$ denote the union of all these $\calF(Q')$.  Note that $\calF_1$ are composed of stopping time regions of each $\calF$ generated by Proposition \ref{p:3.6}, not the union of these stopping time regions.

  By construction, all the properties besides (e) are satisfied.  See the proof of Proposition 4.2 of \cite{semmes} for more information if needed.  The proof of Property (e) is also similar to the proof of the analogous property in \cite{semmes} with only superficial modifications.  For example, let $\calG$ denote the set of cubes for which Proposition \ref{p:3.6} was applied in the above construction and $Q(G)$ denote the set of cubes that make up the stopping time process starting at $G \in \calG$.  Thus, $\calG$ is a subset of $\{Q(S) : S \in \calF_1\}$ and does not contain cubes for which $\mu(Q \cap K) < \eta \mu(Q)$.  We can get the following claim.

  \begin{customclaim}{B.4.16'}
    For each cube $R \in \Delta$, there is a measurable subset $F(R)$ of $R \cap K$ such that $\mu(F(R)) \geq \alpha \mu(R \cap K)$ and so that for each $y \in F(R)$ there is at most one $Q \in \calG$
  \end{customclaim}
  The proof requires only superficial modifications of the proof of the original Claim 4.16 in \cite{semmes}.  The set $F(R)$ is constructed from the modified $E(T_i)$ sets of Proposition \ref{p:3.6} where $T_i$ are maximal cubes of $\Delta(R)$ and $\calG$ along with additional cubes.  Property (a) of Proposition \ref{p:3.6} gives us our needed lower bound for $\mu(F(R))$.  Note that the modified set $E$ of Proposition \ref{p:3.6} also has the cubes $P_j$, but this is fine as the construction above completely terminates at these cubes.  For more information, see the proof of Claim 4.16 in \cite{semmes}.

  As in \cite{semmes}, Claim B.4.16' and Lemma \ref{l:2.28} show that $\calG$ is $K$-Carleson.  The rest of the proof of Property (e) requires showing that all the cubes of $Q(S)$, not just the ones in $\calG$, are $K$-Carleson.  This follows completely analogously as in \cite{semmes}.  We use that each $\{Q(S) \in \calF_1\} \cap Q(G)$ is uniformly $K$-Carleson by Property (e) of Proposition \ref{p:3.6} and each $G \in \calG$ contains only cubes of high density cubes along with Lemma \ref{l:2.50} to establish Property (e).  The details are left to the reader.
\end{proof}

\subsection{Section 5 of \cite{semmes}}
We recall Definition 5.1 of \cite{semmes}, which says that a stopping time region $S$ is {\it good} if for each $Q \in S$, either all of its children are in $S$ or none of them are.

We will need the following lemma for the next section.
\begin{customlemma}{B.5.2'} \label{l:5.2}
  Suppose $S \subseteq \Delta$ is a good stopping time region.  Let $Q \in S$ and $\{T_i\} \subseteq S$ be a finite family of pairwise-disjoint cubes so that $T_i \subseteq Q$ for all $i$.  Then there exists another finite family of pairwise disjoint cubes $\{W_j\} \subseteq S$ so that $W_j \subseteq Q$ for all $j$, each $W_j$ is disjoint from all the $T_i$, and
  \begin{align*}
    Q \cap K = \left( \bigcup_i (T_i \cap K) \right) \cup \left( \bigcup_j (W_j \cap K) \right).
  \end{align*}
\end{customlemma}
The proof requires only superficial modifications of the proof of the original Lemma 5.2 in \cite{semmes} and so will be omitted.

The following proposition says that we can do the a stopping time region decomposition as in Proposition \ref{p:4.2}, but we can further specify that all the stopping time regions we get are good.  The proof requires decomposing each stopping time region from Proposition \ref{p:4.2} into a collection of good stopping time regions and then verifying that we didn't violate the $K$-Carleson condition of the top cubes.

\begin{customproposition}{B.5.5'} \label{p:5.5}
  Let $Q_0 \in \Delta$ and fix $\delta,\tau,\eta > 0$.  There exists a constant $k_2$ depending on $\delta$ and $\tau$, as well as a family $\calF_2$ of stopping time regions in $\Delta$ and two collections $\{Q_i\}_{i \in I}$ and $\{P_j\}_{j \in J}$ of cubes so that the following are true:
  \begin{enumerate}[(a)]
    \item the $Q_i$'s and $P_i$'s together form a pairwise disjoint collection of subcubes of $Q_0$ and the stopping time regions in $\calF_2$ are pairwise-disjoint as subsets of $\Delta(Q_0)$,
    \item if $R \in \Delta(Q_0)$, then either $R \subseteq Q_i$ for some $i \in I$, $R \subseteq P_j$ for some $j \in J$, or $R \in S$ for some $S \in \calF_2$ (but not more than one),
    \item if $Q, \tilde{Q} \in S$ and $S \in \calF_2$, then
    \begin{align*}
      (1+\tau)^{-2} \frac{|\phi(Q \cap K)|}{\mu(Q)} \leq \frac{|\phi(\tilde{Q} \cap K)|}{\mu(\tilde{Q})} \leq (1+\tau)^2 \frac{|\phi(Q \cap K)|}{\mu(Q)},
    \end{align*}
    \item $|\phi(Q \cap K)| \geq (1+\tau)^{-1} \delta \mu(Q)$ when $Q \in S$, $S \in \calF_2$,
    \item the family of cubes $\{Q(S) : S \in \calF_2\}$ is a $K$-Carleson set with constant $k_2$,
    \item $|\phi(Q_i \cap K)| < \delta \mu(Q_i)$ for all $i \in I$,
    \item each $S \in \calF_2$ is a good stopping time region,
    \item $\mu(P_j \cap K) < \eta \mu(P_j)$ for all $j \in J$.
    \item $\mu(Q \cap K) \geq \eta \mu(Q)$ for all $Q \in S$ and $S \in \calF_2$.
  \end{enumerate}
\end{customproposition}

\begin{proof}
  We run the same exact stopping time region decomposition of the proof of Proposition 5.5 of \cite{semmes} on $\calF_1$ of Proposition \ref{p:4.2} to get a family of good stopping time regions $\calF_2$.  The decomposition is the obvious one where, for some $S \in \calF_1$, we take a maximal good stopping time region $S_0 \subseteq S$ such that $Q(S_0) = Q(S)$.  Then for each $R_i \in b(S_0)$, we take again take a maximal good stopping time region $S_i \subseteq S$ so that $Q(S_i) = R$.  We repeat forever on the bottom cubes that we get or until we run out of bottom cubes.  Thus, Property (g) is satisfied by construction and all other properties besides (e) are satisfied by the properties of $\calF_1$, $Q_i$, and $P_j$ of Proposition \ref{p:4.2}.

  As in \cite{semmes}, we see that a top cube $Q \in \{Q(S) : S \in \calF_2\}$ either belongs to $\{Q(S) : S \in \calF_1\}$ or has a parent that belong to some $S \in \calF_1$, but one of the children of the parent (a sibling of $Q$) does not belong to $S$.  This comes from the good stopping time decomposition of the stopping time regions in $\calF_1$.  From the previous proposition, the cubes that are contained in the former case are $K$-Carleson, so we do not have to worry about them.  For cubes from the latter case, we have that one of the siblings $Q'$ of $Q$ must either be a top cube of some other $\tilde{S} \in \calF_1$ or belong to one of the family $\{Q_i\}_{i \in I}$ and $\{P_j\}_{j \in J}$.

  In the first case, we have from Properties (e) and (h) of Proposition \ref{p:4.2} that $\{Q(S) : S \in \calF_1\}$ are $K$-Carleson and contain only high density cubes.  Thus, by Lemma \ref{l:2.32'}, we get that this group of cubes $Q'$ is also $K$-Carleson.  In the second case, we have that $\{Q_i\}_{i \in I}$ and $\{P_j\}_{j \in J}$ are both Carleson sets because they are composed of disjoint cubes.  Thus, Lemma \ref{l:2.32} shows that this group of cubes $Q'$ are Carleson and so also $K$-Carleson.  This finishes the proof of Property (e), which finishes the proof of the entire proposition.
\end{proof}

\subsection{Section 6 of \cite{semmes}}
We keep the same notation as in the previous sections.  We recall some more notation from \cite{semmes}.  For a cube $Q \in \Delta$, we let
\begin{align*}
  *Q = \bigcup \{T \in \Delta_{j(Q)} : \dist(T,Q) \leq \diam(Q)\}.
\end{align*}
Thus, $\hat{Q} = *Q \cap Q_0$.  Given some $\sigma > 0$, we let
\begin{align}
  \calG(\sigma) = \left\{ Q \in \Delta(Q_0) : (1 + \sigma)^{-1} \frac{|\varphi(Q \cap K)|}{\mu(Q)} \leq \frac{|\varphi(\hat{Q} \cap K)|}{\mu(\hat{Q})} \leq (1 + \sigma) \frac{|\varphi(Q \cap K)|}{\mu(Q)} \right\}. \label{e:G-sigma-defn}
\end{align}
Let $\tau$ be a small number and $\delta,\eta > 0$.  Then we can use Proposition \ref{p:5.5} to get a family $\calF_2$ of stopping time regions in $Q_0$ along with two families of mutually disjoint subcubes $\{Q_i\}_{i \in I}$ and $\{P_j\}_{j \in J}$.  We set
\begin{align*}
  G_2 = \bigcup_{S \in \calF_2} S.
\end{align*}

The following property says that, by taking $\tau$ small enough in Proposition \ref{p:5.5}, we can get that most of the cubes in the good stopping time regions are also in $\calG(\sigma)$.

\begin{customproposition}{B.6.13'} \label{p:6.13}
  Let $\sigma,\delta,\eta > 0$.  If we choose $\tau$ small enough, depending on $\sigma$ and $K$, then $G_2 \backslash \calG(\sigma)$ is a $K$-Carleson set with constant depending only on $\tau,\delta,\eta$, and $K$.
\end{customproposition}

The fact that $G_2$ contains only cubes of high density allows us to transition from $K$-Carleson estimates to Carleson estimates.  Keeping this in mind, most of the proof then requires only superficial modifications of the proof of the original Proposition 6.13 in \cite{semmes}.

We now go quickly over the four reductions of the proof of Proposition \ref{p:6.13}.  For our first reduction, we use Lemma \ref{l:2.50} and Property (e) of Proposition \ref{p:5.5}, to get that it suffices to show that if $\tau$ is sufficiently small, then for any $S \in \calF_2$,
\begin{align*}
  S \backslash \calG(\sigma)
\end{align*}
is $K$-Carleson with a bound depending only on $K$ and $\eta$.

For a $S \in \calF_2$, we define
\begin{align*}
  S' = \{Q \in S : T \in S \text{ whenever } T \in \Delta_{j(Q)} \text{ and } \dist(T,Q) \leq \diam(Q) \}.
\end{align*}
For our second reduction, we use Lemma \ref{l:2.58} and our first reduction to get that it suffices to show that if $\tau$ is sufficiently small, then for every $S \in \calF_2$,
\begin{align*}
  S' \backslash \calG(\sigma)
\end{align*}
is $K$-Carleson with constant depending only on $\tau$, $K$, and $\eta$.

Fix an $S \in \calF_2$.  For our third reduction, we use Lemma \ref{l:2.28} to get that it suffices to show that if $\tau$ is sufficiently small, then for every $Q \in S$ there is a measurable subset $D(Q) \subseteq Q \cap K$ such that
\begin{align*}
  \mu(D(Q)) \geq \gamma \mu(Q \cap K)
\end{align*}
and for each $x \in D(Q)$, there are at most $m$ cubes $R \in S' \backslash \calG(\sigma)$ such that $R \subseteq Q$ and $x \in R$.  Here $m$ and $\gamma$ are positive constants that depend only on $K$.

There is one small point in that in Lemma \ref{l:2.28}, we require the $\mu(D(Q))$ bound for all $Q \in \Delta$ not just $S$, but the ones in $S$ are the only ones we really need.

Fixing some $Q \in S$, we set
\begin{align*}
  \calB_1 = \{ R \in S'\setminus \calG(\sigma) : *R \subseteq Q\},
\end{align*}
and
\begin{align*}
  \calB_2 = \{ R \in S' \setminus \calG(\sigma) : R \subseteq Q \text{ but } *R \not\subseteq Q\}.
\end{align*}
We have the following lemma.
\begin{customlemma}{B.6.27'} \label{l:6.27'}
  There is a constant $C_2$ which depends only on $K$ and $\eta$ so that
  \begin{align*}
    \sum_{R \in \calB_2} \mu(R \cap K) \leq C_2 \mu(Q \cap K).
  \end{align*}
\end{customlemma}
The proof follows easily from the proof of the original Lemma 6.27 and the fact that $Q \in S$ has high density.

Using Lemma \ref{l:6.27'}, we get as in \cite{semmes} that it suffices to prove the following modification of the fourth and final reduction of the proof of Proposition \ref{p:6.13}
\begin{customreduction}{B.6.29'} \label{r:6.29}
  It suffices to show for every $Q \in S$ that if $\tau$ is small enough depending on $\sigma$, $\eta$, and $K$, then there is a measurable subset $E(Q)$ of $Q \cap K$ such that $\mu(E(Q)) \geq \frac{1}{2} \mu(Q \cap K)$ and there are no cubes of $\calB_1$ that intersect $E(Q)$.
\end{customreduction}

To prove the third reduction from Reduction \ref{r:6.29}, one lets $N_2(x)$ for $x \in Q$ denote the number of cubes $R \in \calB_2$ such that $x \in R$.  One then defines
\begin{align*}
  D(Q) = \{x \in E(Q) : N_2(x) < 4C_2 \}
\end{align*}
where $C_2$ is as in Lemma \ref{l:6.27'}.  One then easily gets using Lemma \ref{l:6.27'} that $\mu(D(Q)) \geq \frac{1}{4} \mu(Q \cap K)$ with $m = 4C_2$ for the third reduction.

For more details of these reductions, see the proof of Proposition 6.13 in \cite{semmes}.

The proof of Reduction \ref{r:6.29} itself requires mostly superficial modifications of the proof of the original Reduction 6.29 in \cite{semmes}.  For instance, we start off with the following Vitali type lemma that is the original Lemma 6.39 of \cite{semmes} unmodified:
\begin{customlemma}{B.6.39} \label{l:6.39}
  There is a family $\{R_j\}_{j \in J}$ of elements of $\calB_1$ such that
  \begin{align*}
    *R_i \cap *R_j = \emptyset, \qquad \text{when } i \neq j,
  \end{align*}
  and
  \begin{align*}
    \bigcup_{R \in \calB_1} R \subseteq \bigcup_{j \in J} \lambda R_j,
  \end{align*}
  where $\lambda$ depends only on $K$.
\end{customlemma}
The proof is also unchanged and follows the basic structure of the original Vitali lemma.

We thus get that
\begin{align}
  \bigcup_{R \in \calB_1} (R \cap K) \subseteq \bigcup_{j \in J} (\lambda R_j \cap K). \label{e:vitali}
\end{align}
One easily sees that
\begin{align}
  \mu(\lambda R \cap K) \leq \mu(\lambda R) \overset{\eqref{e:cube-diam} \wedge \eqref{e:cube-growth}}{\leq} C \mu(R), \quad \forall R \in \calB_1, \label{e:high-density}
\end{align}
where $C$ depends on $\lambda > 1$ and $K$.  Thus, we get the following equation that is analogous to (6.47):
\begin{align}
  \mu\left( \bigcup_{R \in \calB_1} (R \cap K) \right) \overset{\eqref{e:vitali}}{\leq} \sum_{j \in J} \mu(\lambda R_j \cap K) \overset{\eqref{e:high-density}}{\leq} C \sum_{j \in J} \mu(R_j). \label{e:R-K-bound}
\end{align}

Note that if $R \in \calB_1$ then one easily gets that $\widehat{R} = *R$ by the definitions.  One can also prove the following lemma.

\begin{customlemma}{B.6.56'} \label{l:6.56}
  If $R \in \calB_1$ and $\tau \leq \min \{1,\sigma/3\}$, then
  \begin{align*}
    \frac{|\phi(*R \cap K)|}{\mu(*R)} &< (1+\sigma)^{-1} (1+\tau)^2 \frac{|\phi(Q \cap K)|}{\mu(Q)}.
  \end{align*}
\end{customlemma}

\begin{proof}
  This follows from Property (c) of Proposition \ref{p:5.5} and the fact that $R$ and $Q$ are both in $S$ once we prove that
  \begin{align}
    \frac{|\phi(*R \cap K)|}{\mu(*R)} < (1+\sigma)^{-1} \frac{|\phi(R \cap K)|}{\mu(R)}. \label{e:6.56-reduce}
  \end{align}
  Assume this were not the case.  Then as $R \notin \calG(\sigma)$, we must have that 
  \begin{align}
    \frac{|\phi(*R \cap K)|}{\mu(*R)} > (1+\sigma) \frac{|\phi(R \cap K)|}{\mu(R)}. \label{e:6.56-hypothesis}
  \end{align}
  Let $N(R)$ be the set of cubes in $\Delta_{j(R)}$ such that $*R = \bigcup_{T \in N(R)} T$.  As $R \in \calB_1$, $R$ is also in $S'$.  Thus, by definition of $S'$, we have that $N(R) \subseteq S$ and so
  \begin{align}
    \frac{|\phi(T \cap K)|}{\mu(T)} \leq (1+\tau)^2 \frac{|\phi(R \cap K)|}{\mu(R)}, \qquad \forall T \in N(R), \label{e:N(R)-compare}
  \end{align}
  by Property (c) of Proposition \ref{p:5.5}.  Then
  \begin{multline}
    |\phi(*R \cap K)| \leq \sum_{T \in N(R)} |\phi(T \cap K)| \\
    \overset{\eqref{e:N(R)-compare}}{\leq} \sum_{T \in N(R)} (1 + \tau)^2 \frac{|\phi(R \cap K)|}{\mu(R)} \mu(T)  = (1+\tau)^2 \frac{|\phi(R \cap K)|}{\mu(R)} \mu(*R). \label{e:density-decomp}
  \end{multline}
  As we have set $\tau$ small enough, we see that $(1+\tau)^2 \leq 1+\sigma$ and so this contradicts \eqref{e:6.56-hypothesis}.  Thus, \eqref{e:6.56-reduce} must be true.
\end{proof}

Define
\begin{align*}
  V = \bigcup_{j \in J} *R_j.
\end{align*}
Using Lemma \ref{l:6.56} and an estimate similar to \eqref{e:density-decomp}, we can easily get that
\begin{align}
  |\phi(V \cap K)| \leq (1+\sigma)^{-1}(1+\tau)^2 \frac{|\phi(Q \cap K)|}{\mu(Q)} \mu(V). \label{e:V-small}
\end{align}

We also have the following lemma.
\begin{customlemma}{B.6.61'} \label{l:6.61}
  \begin{align}
    |\phi((Q \cap K) \backslash V)| \leq (1+ \tau)^2 \frac{|\phi(Q \cap K)|}{\mu(Q)} \mu(Q \backslash V). \label{e:QV-also-small}
  \end{align}
\end{customlemma}

\begin{proof}
  Let $J_0$ be any arbitrary finite subset of $J$ and set $V_0 = \bigcup_{j \in J_0} *R_j$.  Let $N(R)$ be as in the proof of Lemma \ref{l:6.56}.  As before, as each $R_j \in \calB_1$, we have that $N(R_j) \subseteq S$.  Apply Lemma \ref{l:5.2} to the finite set $\{T_i\} = \bigcup_{j \in J_0} N(R_j)$ to get a finite collection $\{W_\ell\}$ of disjoint subcubes of $Q$ such that each $W_\ell$ lies in $S$ and
  \begin{align}
    (Q \cap K) \backslash V_0 = \bigcup_\ell (W_\ell \cap K). \label{e:W-equal}
  \end{align}
  Note also that as the $W_\ell \subseteq Q$ and is disjoint from $\{T_i\}$, we have
  \begin{align}
    \bigcup_\ell W_\ell \subseteq Q \backslash V_0. \label{e:W-subset}
  \end{align}
  Because each $W_\ell$ lies in $S$, we get by Property (c) of Proposition \ref{p:5.5} that
  \begin{align}
    |\phi(W_\ell \cap K)| \leq (1+\tau)^2 \frac{|\phi(Q \cap K)|}{\mu(Q)} \mu(W_\ell). \label{e:W-density}
  \end{align}
  This gives
  \begin{multline*}
    |\phi((Q \cap K) \backslash V_0)| \overset{\eqref{e:W-equal}}{\leq} \sum_\ell |\phi(W_\ell \cap K)| \overset{\eqref{e:W-density}}{\leq} (1+\tau)^2 \frac{|\phi(Q \cap K)|}{\mu(Q)} \sum_\ell \mu(W_\ell) \\
    \overset{\eqref{e:W-subset}}{\leq} (1+\tau)^2 \frac{|\phi(Q \cap K)|}{\mu(Q)} \mu(Q \backslash V_0).
  \end{multline*}
  As $V_0 \subseteq V$, we have shown that
  \begin{align*}
    |\phi((Q \cap K) \backslash V)| \leq (1+\tau)^2 \frac{|\phi(Q \cap K)|}{\mu(Q)} \mu(Q \backslash V_0).
  \end{align*}
  As this holds for any finite subset $J_0$ of $J$, we can then ``pass to the limit'' to prove the lemma.
\end{proof}

The two estimates \eqref{e:V-small} and \eqref{e:QV-also-small} gives
\begin{align*}
  |\phi(Q \cap K)| \leq (1+\sigma)^{-1}(1+\tau)^2 \frac{|\phi(Q \cap K)|}{\mu(Q)} \mu(V) + (1+\tau)^2 \frac{|\phi(Q \cap K)|}{\mu(Q)} \mu(Q \backslash V),
\end{align*}
from which we can get (with a little work)
\begin{align*}
  \mu(V) \leq \frac{\tau}{\sigma} 3(1+\sigma) \mu(Q).
\end{align*}
One then gets
\begin{multline*}
  \mu\left( \bigcup_{R \in \calB_1} (R \cap K) \right) \overset{\eqref{e:R-K-bound}}{\leq} C \sum_{j \in J} \mu(R_j) \leq C \sum_{j \in J} \mu(*R_j) \leq C \mu(V) \\
  \leq \frac{C\tau}{\sigma} 3(1+\sigma) \mu(Q) \leq \frac{C\tau}{\eta \sigma} 3(1+\sigma) \mu(Q \cap K).
\end{multline*}
In the last inequality, we used the fact that $Q \in S \in \calF_2$ has high density.  Taking $\tau$ smaller than $\eta \sigma/6C(1+\sigma)$ finishes the proof.

\subsection{Section 7 of \cite{semmes}: Final proof}
We keep the same notation as the previous sections.  Choose $\sigma$ and $\tau$ small enough so that
\begin{align*}
  (1+\sigma)(1+\tau)^2 \leq 1+\zeta
\end{align*}
and $\tau$ is small enough compared to $\sigma$ for the hypothesis of Proposition \ref{p:6.13}.  We need the following lemma.
\begin{customlemma}{B.7.11'} \label{l:7.11}
  $\Delta(Q_0) \backslash G_2 \subseteq \SI(\delta) \cup \LD(\eta)$.
\end{customlemma}
This follows directly from Properties (b), (f), and (h) of Proposition \ref{p:5.5}.  One then gets from Lemma \ref{l:7.11} that, to prove Proposition \ref{p:M-carleson}, it suffices to show that $G_2 \backslash M_A(\zeta)$ is a $K$-Carleson set.

We let $\calG(\sigma)$ be as before and so we get from Proposition \ref{p:6.13} that $G_2 \backslash \calG(\sigma)$ is $K$-Carleson.  For each $S \in \calF_2$, we let $S_A$ denote the set of cubes $Q \in S$ such that every neighbor of $Q$ also lies in $S$.  Then by Lemma \ref{l:2.58} and Properties (e) and (i) of Proposition \ref{p:5.5}, we have that $\bigcup_{S \in \calF_2} (S \backslash S_A)$ is $K$-Carleson.

Also note that if $Q \in S_A$ but there exists some neighbor of $R \in \Delta(Q_0)$ of $Q$ such that $R \notin \calG(\sigma)$ (but $R \in S$ as $Q \in S_A$), then $Q$ is a neighbor of a cube in $G_2 \backslash \calG(\sigma)$.  Thus, Lemma \ref{l:2.32'} gives that this set of $Q$ is $K$-Carleson.

As $G_2$ is the union of $\calF_2$, we see from the previous two paragraphs that we reduce the proof of Proposition \ref{p:M-carleson} to proving the following:
\begin{customreduction}{B.7.15'} \label{r:7.15}
  Let $Q$ be a cube in $G_2$ such that $Q \in S_A$ for some $S \in \calF_2$ and $R \in \calG(\sigma)$ whenever $R \in \Delta(Q_0)$ is a neighbor of $Q$.  Then $Q \in M_A(\zeta)$.
\end{customreduction}

This follows easily from the hypothesis of the reduction and the properties of Proposition \ref{p:5.5}.  The first property of $M_A(\zeta)$ comes from Property (d) of Proposition \ref{p:5.5}.  The second property comes from the fact that since $Q \in S_A$, then if $R$ is a neighbor of $Q$, $R \in S$ and so we get our needed property from Property (c) of Proposition \ref{p:5.5}.  The final property follows from the second property and the fact that $R \in \calG(\sigma)$.  This finishes the proof Reduction \ref{r:7.15}, which finishes the proof of Proposition \ref{p:M-carleson}.


\begin{bibdiv}
\begin{biblist}

\bib{bate}{article}{
  title = {Structure of measures in Lipschitz differentiability spaces},
  author = {Bate, D.},
  journal = {J. Amer. Math. Soc.},
  note = {To appear},
}

\bib{bourdon-pajot}{article}{
  title = {Poincar\'e inequalities and quasiconformal structure on the boundary of some hyperbolic buildings},
  author = {Bourdon, M.},
  author = {Pajot, H.},
  journal = {Proc. Amer. Math. Soc.},
  volume = {127},
  number = {8},
  pages = {2315-2324},
  year = {1999},
}

\bib{cheeger}{article}{
  title = {Differentiability of Lipschitz functions on metric measure spaces},
  author = {Cheeger, J.},
  journal = {Geom. Funct. Anal.},
  volume = {9},
  number = {3},
  pages = {428-517},
  year = {1999},
}

\bib{cheeger-kleiner-rnp}{article}{
    AUTHOR = {Cheeger, J.},
    AUTHOR = {Kleiner, B.},
     TITLE = {Differentiability of {L}ipschitz maps from metric measure
              spaces to {B}anach spaces with the {R}adon-{N}ikod\'ym
              property},
   JOURNAL = {Geom. Funct. Anal.},
  FJOURNAL = {Geometric and Functional Analysis},
    VOLUME = {19},
      YEAR = {2009},
    NUMBER = {4},
     PAGES = {1017--1028},
      ISSN = {1016-443X},
     CODEN = {GFANFB},
   MRCLASS = {30L05 (46G05 58C20)},
  MRNUMBER = {2570313 (2011c:30138)},
MRREVIEWER = {Jeremy T. Tyson},
       DOI = {10.1007/s00039-009-0030-6},
       URL = {http://dx.doi.org/10.1007/s00039-009-0030-6},
}

\bib{cheeger-kleiner}{article}{
  title = {Inverse limit spaces satisfying a Poincar\'e inequality},
  author = {Cheeger, J.},
  author = {Kleiner, B.},
  note = {Preprint},
  year = {2013},
}

\bib{christ}{article}{
  title = {A $T(b)$ theorem with remarks on analytic capacity and the Cauchy integral},
  author = {Christ, M.},
  journal = {Colloq. Math.},
  volume = {60/61},
  number = {2},
  pages = {601-628},
  year = {1990},
}

\bib{csornyei-jones}{article}{
  author = {Cs\"{o}rnyei, M.},
  author = {Jones, P.},
  title = {Product Formulas for Measures and Applications to Analysis and Geometry},
  note = {www.math.sunysb.edu/Videos/dfest/PDFs/38-Jones.pdf}
}

\bib{david}{article}{
  title = {Morceaux de graphes Lipschitziens et int\'egrales singuli\`eres sur un surface},
  author = {David, G.},
  journal = {Rev. Mat. Iberoam.},
  volume = {4},
  number = {1},
  pages = {73-114},
  year = {1988},
}

\bib{david-gc}{article}{
  title = {Tangents and rectifiability of Ahlfors regular Lipschitz differentiability spaces},
  author = {David, G. C.},
  journal = {Geom. Funct. Anal.}
  note = {To appear},
}

\bib{david-semmes}{article}{
  title = {Quantitative rectifiability and Lipschitz mappings},
  author = {David, G.},
  author = {Semmes, S.},
  journal = {Trans. Amer. Math. Soc.},
  volume = {337},
  number = {2},
  year = {1993},
  pages = {855-889},
}

\bib{jones}{article}{
  title = {Lipschitz and bi-Lipschitz functions},
  author = {Jones, P.},
  journal = {Rev. Mat. Iberoam.},
  volume = {4},
  number = {1},
  year = {1988},
  pages = {115-121},
}

\bib{kirchheim}{article}{
  title = {Rectifiable metric spaces: local structure and regularity of the Hausdorff measure},
  author = {Kirchheim, B.},
  journal = {Proc. Amer. Math. Soc.},
  volume = {121},
  number = {1},
  pages = {113-123},
  year = {1994},
}

\bib{laakso}{article}{
  title = {Ahlfors $Q$-regular spaces with arbitrary $Q > 1$ admitting weak Poincar\'e inequality},
  author = {Laakso, T.},
  journal = {Geom. Funct. Anal.},
  volume = {10},
  number = {1},
  pages = {111-123},
  year = {2000},
}

\bib{mathe}{misc}{
  title = {Disintegrating measures onto Lipschitz curves and surfaces},
  author = {M\'ath\'e, A.},
  note = {ERC Workshop on Geometric Measure Theory, Analysis in Metric Spaces and Real Analysis, October 2013},
}

\bib{mattila}{article}{
        AUTHOR = {Mattila, P.},
        TITLE = {Hausdorff {$m$} regular and rectifiable sets in {$n$}-space},
        JOURNAL = {Trans. Amer. Math. Soc.},
        FJOURNAL = {Transactions of the American Mathematical Society},
        VOLUME = {205},
        YEAR = {1975},
        PAGES = {263--274},
        ISSN = {0002-9947},
        MRCLASS = {28A75},
        MRNUMBER = {0357741 (50 \#10209)},
        MRREVIEWER = {S. J. Taylor},
}

\bib{pansu}{article}{
  title = {M\'etriques de Carnot-Carath\'eodory et quasiisom\'etries des espaces sym\'etriques de rang un},
  author = {Pansu, P.},
  journal = {Ann. of Math. (2)},
  volume = {129},
  number = {1},
  pages = {1-60},
  year = {1989},
}

\bib{preiss}{article}{
    AUTHOR = {Preiss, D.},
     TITLE = {Geometry of measures in {${\bf R}\sp n$}: distribution,
              rectifiability, and densities},
   JOURNAL = {Ann. of Math. (2)},
  FJOURNAL = {Annals of Mathematics. Second Series},
    VOLUME = {125},
      YEAR = {1987},
    NUMBER = {3},
     PAGES = {537--643},
      ISSN = {0003-486X},
     CODEN = {ANMAAH},
   MRCLASS = {28A75},
  MRNUMBER = {890162 (88d:28008)},
MRREVIEWER = {K. J. Falconer},
       DOI = {10.2307/1971410},
       URL = {http://dx.doi.org/10.2307/1971410},
}

\bib{preiss-tiser}{article}{
    AUTHOR = {Preiss, D.},
    AUTHOR = {Ti{\v{s}}er, J.},
     TITLE = {On {B}esicovitch's {$\frac12$}-problem},
   JOURNAL = {J. London Math. Soc. (2)},
  FJOURNAL = {Journal of the London Mathematical Society. Second Series},
    VOLUME = {45},
      YEAR = {1992},
    NUMBER = {2},
     PAGES = {279--287},
      ISSN = {0024-6107},
     CODEN = {JLMSAK},
   MRCLASS = {28A75 (28A78)},
  MRNUMBER = {1171555 (93d:28012)},
MRREVIEWER = {Hermann Haase},
       DOI = {10.1112/jlms/s2-45.2.279},
       URL = {http://dx.doi.org/10.1112/jlms/s2-45.2.279},
}

\bib{semmes}{article}{
  title = {Measure-preserving quality within mappings},
  author = {Semmes, S.},
  journal = {Rev. Mat. Iberoam.},
  volume = {16},
  number = {2},
  year = {2000},
  pages = {363-458},
}

\end{biblist}
\end{bibdiv}
\end{document}